\documentclass{amsart}
\usepackage{amsmath, amsthm, amssymb,pdfsync}

\input xy
\xyoption{all}

\swapnumbers
\theoremstyle{plain} 
\newtheorem{theorem}[subsection]{Theorem}
\newtheorem{proposition}[subsection]{Proposition}
\newtheorem{lemma}[subsection]{Lemma}
\newtheorem{corollary}[subsection]{Corollary}

{\swapnumbers \theoremstyle{plain} 
\newtheorem{theorem*}{Theorem}
\theoremstyle{definition}
\newtheorem{conjecture*}[section]{Conjecture}
}

\theoremstyle{definition}

\newtheorem{example}[subsection]{Example}
\newtheorem{definition}[subsection]{Definition}
\newtheorem{question}[subsection]{Question}
\theoremstyle{remark}
\newtheorem{remark}[subsection]{Remark}

\newcommand\Aut{\text{\rm Aut}}

\newcommand{\mpair}[1]{\pair{\,#1\,}}

\newcommand{\mset}[1]{\set{\,#1\,}}

\newcommand{\pair}[1]{\langle #1\rangle}

\newcommand{\set}[1]{\{#1\}}

\newcommand{\card}[1]{\sharp(#1)}

\newcommand\lex{{\text{\rm lex}}}
\makeatletter
\renewcommand{\p@enumi}{\thesubsection}
\makeatother
\newenvironment{conds}{
                       
                        \begin{enumerate} }
                     {\end{enumerate} }

\newenvironment{num}{
                      
                      \begin{enumerate} }
                    {\end{enumerate} }

                    \newcommand{\op}{^\text{\rm op}}
\newcommand{\Int}{{\mathbb Z}}
\newcommand{\Nat}{{\mathbb N}}
\newcommand{\real}{{\mathbb R}}
\newcommand{\mc}{\mathcal}
\newcommand{\Id}{\text{\rm Id}}
\newcommand{\ck}[1]{{#1}^\vee}

\begin{document}

\title{On rigidity of abstract root systems\\ of Coxeter systems}

\author{Matthew Dyer}
\address{Department of Mathematics 
\\ 255 Hurley Building\\ University of Notre Dame \\
Notre Dame, Indiana 46556, U.S.A.}
\date{\today}

\begin{abstract} We introduce and study a combinatorially defined notion of root basis of a (real) root system of a possibly infinite Coxeter group.
Known results on conjugacy up to sign of root bases of certain irreducible finite rank real root systems  are extended to abstract root bases, to a larger class of  real root systems, and, with a short list of (genuine) exceptions, to infinite rank irreducible Coxeter systems.    
\end{abstract}

\maketitle


\section*{Introduction} 

It is well known that any two root bases (simple systems of roots) for a root system of a finite
Weyl group (or finite Coxeter group) $W$ are $W$-conjugate.
This result has been extended to $W$-conjugacy up to sign of root bases of root systems of 
certain   reflection representations of finite rank, irreducible Coxeter systems in \cite{How}
and \cite{Moop}
(see also \cite[Proposition 5.9]{Kac}).

In this paper, we extend these results in three ways.
First, we reformulate  the results  of \cite{How} as  asserting  conjugacy up to sign of
suitably defined  ``abstract root bases'' of ``abstract  root systems''  of irreducible, finite rank Coxeter systems $(W,S)$; the main point here is that we provide  a characterization of root bases which does not require the  linear structure of the ambient real vector space. Secondly, we use the abstract result to 
extend the above conjugacy result   to (real) root systems  of 
more general    reflection representations  of finite rank irreducible Coxeter systems. Thirdly, we prove that   two   root bases for  any  (real or abstract)  root system of an  irreducible Coxeter system of possibly infinite rank are ``locally $W$-conjugate'' up to sign,  with a small number of types of  exceptions.

In order to explain some of these results in more detail, 
let $(W,S)$ be a  Coxeter system, and $\Phi$ be the  
root system of the standard reflection representation of $(W,S)$ (see \cite[Ch 5]{Hum} and  \cite[Ch V]{Bou}), with standard set of positive roots $\Phi_+$ and corresponding standard root basis $\Pi$.
For $\alpha\in \Phi$, the reflection
$s_\alpha$ permutes $\Phi$, and $W$ identifies with the group of permutations of $\Phi$ generated by
 the restrictions of the $s_\alpha$ to $\Phi$ (which we denote still as $s_\alpha$).
We let the sign   group $\set{\pm 1}$ act on $\Phi$
by $(\pm 1)\alpha=\pm \alpha$ for $\alpha\in \Phi$.  
The map $\alpha\mapsto s_\alpha$ is a two-fold covering of its image (the set $T$ of reflections of $(W,S)$) with the orbits of the sign group as fibers. We call the  set $\Phi$ with its   action by $\set{\pm 1}$ and  the map $\alpha\mapsto s_\alpha\colon \Phi\rightarrow \text{\rm Sym}(\Phi)$  the standard abstract root system of $(W,S)$.

We say that a subset $\Psi_+$ of $\Phi$ is a quasi-positive system
if it is a set of orbit representatives for the sign group acting on $\Phi$.
Then  $\alpha\in \Psi_+$ is said to be a simple root  of  $\Psi_+$ if $s_\alpha$ permutes
$\Psi_+\setminus\set{\alpha}$. Let $\Pi'$ be the set of simple roots of $\Psi_+$,
$S'=\mset{s_\alpha\mid \alpha\in \Pi'}$ be the set of simple reflections for $\Psi_+$,  and $W'=\mpair{S'}$ be the subgroup generated by $S'$.
It is easy to see that $(W',S')$ is a Coxeter system (we prove this here as \ref{ss:3.4}(c), though it follows also from   \cite[1.8]{Dy1}). We say that $\Psi_+$ is a generative quasi-positive system  if $W'=W$.
In this case, it need not be true that $(W,S)$ is isomorphic to $(W,S')$ as Coxeter system;
indeed,  some (but not all) examples  of non-isomorphic finite rank, irreducible  Coxeter systems  $(W,S)$ and $(W,S')$  related by diagram twisting in the sense of \cite{Mue} arise  in this way (see \ref{ss:3.16}--\ref{ss:3.19}).

We say that  $\gamma\in \Phi$ is between $\alpha\in \Phi$ and $\beta\in \Phi$
iff $\gamma=a\alpha+b\beta$ for some $a,b\in \real_{\geq 0}$.   Say that $\Psi\subseteq \Phi$
is closed if $\alpha,\beta\in \Psi$, and $\gamma\in \Phi$ with $\gamma$ between $\alpha$ and $\beta$ implies $\gamma\in \Phi$.  
We define $\Psi_+\subseteq \Psi$ to be  an abstract positive system for $\Psi$ if 
it is a generative, closed, quasi-positive system. It is easy to see that the above notions of  abstract root system, betweenness, abstract positive system etc may all be reformulated   purely combinatorially and algebraically in terms of $(W,S)$.

Recall the Coxeter system of type $A_\infty$ (resp., $A_{\infty,\infty})$ has as underlying Coxeter  group the group of all permutations of $\Int$ (resp., $\Nat$) leaving all but finitely many elements fixed, with Coxeter generators given by the adjacent transpositions $(n,n+1)$ for $n\in \Int$ (resp., $n\in \Nat$).
These Coxeter systems are non-isomorphic, but
any bijection $\Nat\rightarrow \Int$ induces a reflection preserving isomorphism of
the underlying Coxeter groups $W(A_\infty)\xrightarrow{\cong} W(A_{\infty,\infty})$.
We prove in this paper the following results.

\begin{theorem*} Let   $(W,S)$ be an irreducible, Coxeter system which is   not necessarily of finite rank. Let $\Psi_+$ be an  abstract positive system of $\Phi$, and $\Pi'$, $S'$ denote respectively the sets of simple roots and simple reflections  of $\Psi_+$.
Then
\begin{num}   \item $(W,S)$ and $(W,S')$ have the same finite rank parabolic subgroups.
\item The Coxeter system  $(W,S)$ is isomorphic to $(W,S')$  unless perhaps
 one of them is of type $A_{\infty}$ and the other is of type $A_{\infty,\infty}$.
\item If  $(W,S)$ and $(W,S')$ are isomorphic, then there is a permutation 
$\sigma$ of  $\Phi$ and a sign $\epsilon\in \set{\pm 1}$ with $\Psi_+=\epsilon \sigma(\Phi_+)$ and $\Pi'=\epsilon \sigma(\Pi)$, such that for
 any subgroup $W'$ of $W$ which is generated by a finite subset of $T$, there exists $w=w(W')\in W$
 such that $\sigma(\alpha)=w(\alpha)$ for all $\alpha\in \Phi$ with $s_\alpha\in W'$. 
\item If $(W,S)$ is of finite rank, then in $\text{\rm (c)}$, $\sigma\in W$, and moreover, $\sigma$ and $\epsilon$ are then  uniquely determined provided $\epsilon=1$ if $W$ is finite. 
  \item A subset $\Delta $ of $\Phi$ is a root basis of $\Phi$ iff it is the set of abstract simple roots of some abstract positive system of $\Phi$.
  \end{num}\end{theorem*}
  
  In the body of the paper, we extend the theorem to a more  general class of root systems   (see Section 2)  parametrized roughly by what we call  possibly non-integral generalized Cartan matrices
(NGCM's).  Most (reduced, real) root systems  usually considered in the literature 
(e.g. in \cite{Bou}, \cite{Hum}, \cite{Kac}, \cite{How}, \cite{Dy2})
  are in this class and  it has the important technical advantage of being  closed under passage to root subsystems for arbitrary reflection subgroups. In general, for some of the reflection representations considered, it is not possible to choose a ``reduced'' root system,  but we do not restrict 
 to the subclass of  ``reduced'' root systems since this would amount to imposing an unnatural
condition on the NGCM (compare \ref{ss:4.5}, \ref{ss:4.7}  and \cite{Jon}).  Since,  for non-reduced root systems,  a given set of simple reflections may correspond to several different sets of simple roots differing by ``rescaling'' (multiplying each by a positive scalar), the statement of the theorems need minor modification for non-reduced root systems; for instance, conjugacy results hold only up to sign and rescaling.    The restriction to irreducible Coxeter systems   in (a)--(d) is only a matter of convenience, since those parts of the theorem can be applied separately to the irreducible components; the situation for (e) in general is more complicated (see \ref{ss:5.2} and \ref{ss:5.8}). 
  
The arrangement of the paper is as follows. Section 1  contains  
definitions of,  and basic facts about,  abstract root systems of Coxeter systems.
Some of the   facts    involving the relation of the root system to the  ``reflection cocycle''  extend  
to  similar structures (which we call quasi-root systems, to avoid confusion) attached to other groups.
Section 2 collects basic properties of the  real reflection representations of Coxeter groups, and corresponding root systems, which we consider in this paper.  The main results, involving the connection between the abstract root systems of Section 1 and  the real root systems of Section 2, and including the statements and  proof of  more general versions of the parts of Theorem 1, are given in Section 3. 

There are two appendixes dealing with matters subsidiary to our main concerns.  Appendix A 
 recalls some definitions and  basic properties of possibly non-abelian cohomology which  provide the natural setting for some of the special results proved  in  Section 1 involving  reflection cocycles. In particular, the generalities in Appendix A are relevant to classification of abstract root systems, though we do not pursue this here.  
 Appendix B gives some examples and further results involving  general
 quasi-root systems, of which we make no essential use in the body of the paper.
 We observe in particular  that  the definition of  Bruhat orders (and their twisted versions \cite{Dy1}) in terms of the reflection cocycle of a Coxeter system extends to other groups with quasi-root systems realized linearly over real vector spaces, (e.g.  real orthogonal groups).

\section{Abstract root systems}
  \subsection{} \label{ss:2.1}Let $\Psi$ be a set with a given function
   $F  \colon \Psi\rightarrow \text{\rm Sym}(\Psi)$ where $\text{\rm Sym}(\Psi)$
   is the symmetric group consisting of all permutations of $\Psi$. Set $s_\alpha:= F(\alpha)$. We make the following assumptions:
   \begin{conds}\item there is a fixed point free action of the sign group $\set{\pm 1}$ on
   $\Psi$ such that for $\alpha\in \Psi$, $(-1)\alpha=-\alpha:=s_\alpha(\alpha)$ and $s_{-\alpha}=s_\alpha$
   \item  $s_{s_\alpha(\beta)}=s_\alpha s_\beta s_\alpha$ for all $\alpha,\beta\in \Psi$.\end{conds}
    The conditions imply that $s_\alpha$ is an involution (i.e. it is a permutation of  order exactly $2$)
commuting with the action of the sign group,

    To avoid any possible confusion with abstract root systems of Coxeter groups, which
    may naturally be regarded as examples of this formalism (see \ref{ss:3.1})  we call such  a pair $(\Psi,F)$ a quasi-root system; other examples of quasi-root systems  are mentioned in Appendix B.  By abuse of notation, we sometimes call $\Psi$ itself a quasi-root system.
 \subsection{} \label{ss:2.1a}    A morphism $(\Psi^1,F^1)\rightarrow (\Psi^2,F^2)$ of quasi-root systems is defined to  be function
    $\theta\colon \Psi^1\rightarrow \Psi^2$ such that  $\theta(s_\alpha(\beta))=s_{\theta(\alpha)}(\theta(\beta))$ for all $\alpha,\beta\in \Psi^1$. Taking $\alpha=\beta$ shows that  
       $\theta(-\alpha)=-\theta(\alpha)$ for $\alpha\in \Psi_1$.
  
     With the obvious definition of composition of morphisms, the quasi-root systems form a category.
     In particular, this gives rise to the usual (categorical)  notions of an isomorphism of quasi-root systems, and  of the action of a group $G$ on a quasi-root system; an isomorphism is a bijective  morphism of quasi-root systems, and an action of a group $G$  on a quasi-root system $\Psi$  is a homomorphism from
     $G$ to the group $\Aut(\Psi)$ of automorphisms of $\Psi$.  An arbitrary family $\set{(\Psi^i,F^i)}_{i\in I}$ 
     of quasi-root systems has a ``union''   $(\Psi,F)$ where $\Psi=\coprod_{i\in I}\Psi^i$
      (coproduct in the category of sets i.e. disjoint union) and for $\alpha\in \Psi^i\subseteq \Psi$,
      $F(\alpha)_{\vert \Psi^i}=F^i(\alpha)$ while  $F(\alpha)_{\vert \Psi^j}=\Id_{\Psi^j}$ for $j\neq i$.
    
  \subsection{} \label{ss:2.2}   
       Till \ref{ss:2.5e}, we fix a quasi-root system $(\Psi,F)$  and let $G$ be a group acting on $\Psi$ by means 
       of a homomorphism $\iota\colon G\rightarrow G_\Psi:=\Aut(\Psi)$. Set $T=T_\Psi:=\mset{s_\alpha\mid \alpha\in \Psi}\subseteq G_\Psi$ and let $W=W_\Psi:=\mpair{T}$ be the subgroup of $G_\Psi$ generated by $T$. Note $G$ acts on the left on $T$ by $(g,t)\mapsto \iota(g)t\iota(g)^{-1}$.  By abuse of notation, we write this as $(g,t)\mapsto gtg^{-1}$. Then 
       \begin{equation} \label{eq:1} s_{\sigma(\alpha)}=\sigma s_\alpha \sigma^{-1},\quad \sigma\in G, \alpha\in \Psi\end{equation}
   Regard the power set $\mc{P}(T)$ as an additive abelian group under symmetric difference
   $A+B=(A\cup B)\setminus (A\cap B)$, and as a $G$-module with $G$ acting by conjugation ($\sigma\cdot A=\sigma A\sigma^{-1}:=\set{\sigma t\sigma^{-1}\mid t\in A}$ for $\sigma\in G$ and $A\subseteq T$).

    Define a  map $\tau=\tau_{\Psi}\colon \Psi\rightarrow T$  by $\tau(\alpha)=s_\alpha$.
We call $\tau$  the reflection map of $\Psi$.  The sets  $\tau^{-1}(t)$ for $t\in T$ are called the fibers of $\tau$. 

 \subsection{} Consider  two  sets $\Psi_+, \Psi'_+$ of orbit representatives for the sign group
   on $\Psi$ i.e. $\Psi_+$ is a subset of $\Psi$ with $-\Psi_+=\Psi\setminus \Psi_+$, and similarly for
   $\Psi'_+$.
         We say that $\Psi_+$ and $\Psi_+'$  are compatible
  if for each $t\in T$, $\Psi_+\cap \tau^{-1}(t)=\pm(\Psi'_+\cap \tau^{-1}(t))$ or  equivalently, if $T$ is the disjoint union  \[T=\mset{s_\alpha\mid \alpha\in \Psi_+\cap -\Psi'_+}\cup \mset{s_\alpha\mid \alpha\in \Psi_+\cap \Psi'_+}.\]
  Note that $\Psi_+$ and $-\Psi_+$ are compatible. Compatibility is an equivalence relation on the family of sets of orbit representatives for $\set{\pm 1}$ on $\Psi$. 

  We say that $\Psi_+$ is a quasi-positive system for $G$ on $\Psi$ iff
  $\Psi_+$ and $\sigma(\Psi_+)$ are compatible for all $\sigma\in G$. It is not clear that such a quasi-positive system need exist in general, but it will in cases of concern in this paper. In particular,   if  $\set{\pm 1}$
  acts simply transitively on each fiber of the reflection map, then any set of orbit representatives for $\set{\pm 1}$ is a quasi-positive system for $G$ on $\Psi$, and any two quasi-positive systems for $G$ on $\Psi$ are 
  compatible.
  
   \begin{proposition} \label{ss:2.4} Let $\Psi_+$ be a  quasi-positive system for $G$ on $\Psi$.
    Define the function $N_{\Psi_+}\colon G\rightarrow \mc{P}(T)$ by 
   $N_{\Psi_+}(\sigma)=\mset{s_\alpha\mid \alpha\in\Psi_+\cap \sigma (-\Psi_+)}$.

    \begin{num} \item The function  $N_{\Psi_+}$ is a cocycle on $G$ with values in $\mc{P}(T)$.
    \item $T_{\Psi_+}:=T\times \set{\pm 1}$ is a $G$-set  with the action   \begin{equation*} \sigma(t,\epsilon)=(\sigma t\sigma^{-1},\eta(\sigma^{-1},t)\epsilon), \quad
   \eta(\sigma,t):=\begin{cases} 1,&t\not \in N_{\Psi_+}(\sigma)\\ -1,& t\in N_{\Psi_+}(\sigma).\end{cases}\end{equation*}
   \item Define the map $F_{\Psi_+}\colon T_{\Psi_+}\rightarrow \text{\rm Sym}(T_{\Psi_+})$
   by $(t,\epsilon)\mapsto (\alpha\mapsto t(\alpha))$, regarding $T_{\Psi_+}$ as $G_\Psi$-set via $\text{\rm (b)}$. Then $(T_{\Psi_+},F_{\Psi_+})$ is a quasi-root system for $G$. Further,  $\set{\pm 1}$ acts simply transitively on each fiber of the reflection map for $T_{\Psi_+}$ 
and $T\times\set{1}$
   is a quasi-positive system for $G$ on $T_{\Psi_+}$   \item The map $\rho\colon \Psi\rightarrow T_{\Psi_+}$ defined by $\alpha\mapsto (s_\alpha,\epsilon)$ for $\epsilon\in\set{\pm 1}$ and  $\alpha\in \epsilon \Psi_+$ is a surjective morphism of $G$-sets. Further, the map $\Psi'_{+}\mapsto \rho(\Psi_+')$ is  a bijection between the sets of  quasi-positive systems of $\Psi$ compatible with $\Psi_+$ and the quasi-positive systems for  $T_{\Psi_+}$. Finally, $\rho$ is an isomorphism if $\set{\pm 1}$ acts simply transitively on each fiber of the reflection map for $\Psi$.  
   \end{num}\end{proposition}
     \begin{proof}  Write $N$ for $N_{\Psi_+}$. For (a), we have to check the cocycle condition (see \ref{ss:1.1})
   \begin{equation*}N(\sigma\tau)=N(\sigma)+\sigma N(\tau)\sigma^{-1},\qquad \sigma,\tau\in G.\end{equation*}
   Let $\beta\in \Psi_+$. Then, using compatibility of $\Psi_+$ with $\pm \rho\Psi_+$ for all $\rho\in G$, \begin{equation*}\begin{split}&s_\beta\in N(\sigma\tau)\Longleftrightarrow 
   (\tau^{-1}\sigma^{-1})(\beta)\in -\Psi_+\Longleftrightarrow\\& 
   \bigl(\sigma^{-1}(\beta)\in \Psi_+,\  (\tau^{-1}\sigma^{-1})(\beta)\in -\Psi_+
  \bigr) \text{ \rm or }\bigl(\sigma^{-1}(\beta)\in -\Psi_+,\  (\tau^{-1}\sigma^{-1})(\beta)\in -\Psi_+\bigr)
 \\& \Longleftrightarrow\ 
 \bigl(s_\beta\not\in N(\sigma)\text{ \rm and }s_{\sigma^{-1}(\beta)}\in N(\tau)\bigr)  
    \text{ \rm or } \bigl(s_\beta\in N(\sigma)\text{ \rm and }s_{\sigma^{-1}(\beta)}\not\in N(\tau)\bigr)\\
  & \Longleftrightarrow s_\beta\in N(\sigma)+\sigma N(\tau)\sigma^{-1}\end{split}\end{equation*} since 
   $s_{\sigma^{-1}(\beta)}\in N(\tau)$ iff $s_\beta\in \sigma N(\tau)\sigma^{-1}$.
  This proves (a). Then (b) follows from (a) by a straightforward calculation
  (more generally, see \ref{ss:1.4} and its proof, noting    the identification  $\mc{P}(T)\cong \prod_{t\in T}\set{\pm 1}$ of groups with $G$-action).  Finally,  (c)--(d) follow from (b) and the definitions.
  \end{proof} 

   \begin{proposition} \label{ss:2.5} Let $\Psi_+$ and $\Psi'_+$ be compatible quasi-positive systems for $\Psi$.
   \begin{num}  \item Let $A:=\mset{s_\alpha\mid \alpha\in \Psi_+\cap -\Psi'_+}$.
   Then for $\sigma\in G$, \[N_{\Psi_+}(\sigma)+\sigma A\sigma^{-1}=\mset{s_\alpha\mid \alpha\in \Psi_+\cap \sigma(-\Psi'_+)}=N_{\Psi'_+}(\sigma)+A.\]
   \item The cocycles $N_{\Psi_+}$ and $N_{\Psi_+'}$ are cohomologous i.e. the cohomology classes of $N_{\Psi_+}$ and $N_{\Psi'_+}$ in $H^1(G,\mc{P}(T))$ are equal.
   \item The map $T_{\Psi_+}\rightarrow T_{\Psi'_+}$ given by $(s_\alpha,\epsilon)\mapsto (s_\alpha,\epsilon \epsilon')$ for  $\epsilon, \epsilon'\in \set{\pm 1}$ and   $\alpha\in \Psi_+\cap \epsilon' \Psi'_+$ is a $G$-equivariant  isomorphism of quasi-root systems. 
    \end{num}\end{proposition}
  \begin{proof}  The proof that $\mset{s_\alpha\mid \alpha\in \Psi_+\cap \sigma(-\Psi'_+)}=N_{\Psi_+}(\sigma)+\sigma A\sigma^{-1}$ is similar to that of \ref{ss:2.4}(a).
  By symmetry,  $\mset{s_\alpha\mid \alpha\in \sigma^{-1}(\Psi_+)\cap -\Psi'_+}=N_{\Psi'_+}(\sigma^{-1})+\sigma^{-1} A\sigma$, which implies the other equality in (a). Note that  (a) provides an explicit cohomology between $N_{\Psi_+}$ and $N_{\Psi'_+}$, proving (b) (see \ref{ss:1.1}). 
  Then (c) follows from (b) by a simple calculation (see \ref{ss:1.2} more generally). \end{proof}
   \begin{corollary} \label{ss:2.5e} Let $X$ denote the the set of quasi-positive systems for the quasi-root system
 $T_{\Psi_+}$, with $G$-action $(\sigma,\Psi'_+)\mapsto \sigma(\Psi'_+)$.
 Then the map $X\rightarrow \mc{P}(T)$ given by $\Psi'_+\mapsto\mset{t\in T\mid (t,-1)\in \Psi'_+}$ is a $G$-equivariant bijection,  where  $\mc{P}(T)$ has the $G$-action 
 $(g,A)\mapsto N_{\Psi_+}(g)+gAg^{-1}$ (cf $\text{\rm \ref{ss:1.2}}$).
  \end{corollary}
   \begin{proof} This follows  using \ref{ss:2.5}(a).\end{proof}

\subsection{} \label{ss:3.1a}  Now we recall some characterizations of Coxeter systems which 
will be useful to us, and are all either implicit or explicit in the literature. For general references on Coxeter groups, see \cite{Bou}, \cite{Hum} and  \cite{BjBr}.

Consider  a pair $(W,S)$ consisting of a group $W$ and a set $S$
of involutions  generating $W$. The function $l\colon W\rightarrow \Nat$  defined by \[l(w)=\min\mset{n\in \Nat\mid w=s_1\cdots s_n\text{\ for some $s_1,\ldots, s_n\in S$}}\] is called the length function of $(W,S)$.
 Let $T=\cup_{w\in W} wSw^{-1}$ and $\mc{P}(T)$
denote the power set of $T$, regarded as abelian group under symmetric difference
$A+B=(A\cup B)\setminus (A\cap B)$. We let $W$ act on $\mc{P}(T)$ by conjugation
and the group $\set{\pm 1}$ act on $\widehat T:=T\times \set{\pm 1}$ by multiplication on the second factor.
 Set $\widehat T_+=T\times\set{1}\subseteq \widehat T$.
\begin{proposition} \label{ss:3.1} The following conditions are equivalent: 
\begin{conds}\item the  pair $(W,S)$ is a Coxeter system. 
\item there is a  action of the group $W$ by permutations on the set
$\widehat T$ such that $s(t,\epsilon)=(sts,(-1)^{\delta_{s,t}}\epsilon )$
for $s\in S$, $t\in T$, $\epsilon\in\set{\pm 1}$.
\item  there is a function
$N\colon W\rightarrow \mc{P}(T)$ satisfying  
\[N(xy)=N(x)+xN(y)x^{-1}
\text{ for } x,y\in W,\qquad
N(s)=\set{s}\text{ for } s\in S.\] \end{conds} 
Suppose that these conditions hold. One then  has 
\begin{gather*} N(w)=\mset{t\in T\mid l(tw)<l(w)}=\mset{t\in T\mid w^{-1}(t,1)\in -\widehat T_+},\qquad l(w)=\vert N(w)\vert \\
 w(t,\epsilon)=(wtw^{-1},\eta(w,t)\epsilon),\qquad
\eta(w,t)=\begin{cases} 1,&\text{if $t\not\in N(w^{-1})$}\\
-1,&\text{if $t\in N(w^{-1})$.}\end{cases}\end{gather*}
Further, the action of $W$ by permutations on $\Psi:=\widehat T$ is faithful. Set $s_\alpha=t\in \text{\rm Sym}(\Psi)$ for $\alpha=(t,\epsilon)\in \widehat T$, and let  map $F\colon \Psi\rightarrow \text{\rm Sym}(\Psi)$ be the map $\alpha\mapsto s_\alpha$. Then $(\Psi,F)$ is a quasi-root system for $W$ with $\widehat T_+$ as a quasi-positive system. 
\end{proposition} 
 \begin{proof} For completeness, we sketch direct proofs of all the implications amongst (i)--(iii).
  For (i) implies (ii), see \cite[Ch IV, \S 1.4]{Bou}; for the converse, one could  note 
that the arguments of     \cite[Ch IV, \S 1.4--1.5]{Bou} work with minor changes under the assumption (ii) instead of (i). If (ii) holds, it is easy to see that $(\Psi,F)$ is a quasi-root system with $T\times \set{1}$ as a quasi-positive system, and then (ii) implies (iii) follows readily from  Proposition \ref{ss:2.4}(a).  A proof of the equivalence of
   (i) and (iii) can be found in \cite{Dy2} or \cite{Dy0}, though the proof that (i) implies (iii) there goes via
   (ii) (or a similar result).
    A direct proof that (i) implies (iii)
   can be given  by noting that a cocycle on a group $G$ with values in an abelian group can be specified arbitrarily on a set
   of generators  of $G$ provided it preserves (in an obvious sense) a set of defining relations for $G$
   in terms of those generators; it is straightforward then  to deduce existence of the cocycle $N$ as in (iii)  from the standard presentation of $W$.

   A direct proof of the  equivalence of (ii) and (iii) is implicit as a special case of the proof of Proposition \ref{ss:1.4}. 
Specifically,    there is a bijection between the set of  functions  $N\colon W\rightarrow \mc{P}(T)$ and the set of functions   $\eta\colon W\times \widehat T\rightarrow\set{\pm 1}$  given by sending $N$ to $\eta$ as defined
   by the  formula in the statement of the Proposition \ref{ss:3.1}. One checks immediately that
   $N$ is a cocycle iff   $\eta(xy,t)=\eta(x,yty^{-1})\eta(y,t)$ for all $x,y\in W$, $t\in T$ iff
   $w(t,\epsilon)=(wtw^{-1},\eta(w,t)\epsilon)$ defines a representation of $W$ on $\widehat T$.
   Further, for $s\in S$, $N(s)=\set{s}$ iff $s(t,\epsilon)=(sts,(-1)^{\delta_{s,t}}\epsilon )$.
   
   The remaining assertions are clear from the above references and the arguments of the preceding paragraph.\end{proof}
   
\subsection{} \label{ss:3.2a}Whenever the conditions of the proposition hold,
 we call $l$ the standard length function, $T$
the set of reflections, and $N$ the reflection cocycle of $(W,S)$.
 We call the quasi-root system  $(\widehat T, F)$  the
 standard abstract root system of $(W,S)$. 
The sets $\widehat T_+:=T\times\set{1}\subseteq \widehat T$ and $\hat S:=S\times\set{1}\subseteq \widehat T_+$
  are called the standard set of positive roots and standard root basis of $\widehat T$, respectively.
  When convenient, we identify $W$ with a group of permutations of $\widehat T$ in the natural way.

  We take the opportunity to  recall a well-known fact about Coxeter systems, which will be used several times in the sequel.
  \begin{lemma} \label{ss:mingen} For any Coxeter system $(W,S)$, $S$ is a minimal (under inclusion) set of generators of $W$. \end{lemma}

  \subsection{}\label{ss:3.3}  Consider a quasi-root system $(\Psi,F)$  with a specified quasi-positive system 
  $\Psi_+$. 
    If $t=s_\beta$ with $\beta\in \Psi_+$,   then $s_\beta(\beta)=-\beta\in -\Psi_+$. We say that $\beta\in \Psi_+$ is a simple quasi-root  for $\Psi_+$ if for each $\alpha\in \Psi_+$
  with $s_\beta(\alpha)\not\in \Psi_+$, one has $s_\beta=s_\alpha$. 
  Note that one may have distinct simple quasi-roots $\beta$, $\gamma$ for $\Psi$ with $s_\beta=s_\gamma$.
  The reflection $s_\beta$ in a simple quasi-root $\beta$ for $\Psi_+$ is called a simple reflection for $\Psi_+$.

   Let $S'$ be any subset of the set of all simple reflections
  for $\Psi_+$, and let  $\Pi':=\mset{\beta\in \Psi_+\mid s_\beta\in S'}$  be the set of all simple quasi-roots $\beta$ for $\Psi_+$ with $s_\beta\in S'$. Let $W'=\mpair{S'}$ denote the subgroup of $W=\mpair{s_\alpha\mid \alpha\in \Psi}$ generated by $S'$,
  $\Phi':=\mset{w(\alpha)\mid w\in W',\alpha\in \Pi'}$, $T':=\mset{s_\alpha\mid \alpha\in \Phi'}$  and $\Phi'_+:=\Phi'\cap \Psi_+$. By definition of $\Pi'$ and $T'$,  $\Phi'_+=\mset{\alpha\in \Psi_+\mid s_\alpha\in T'}$.

  \begin{proposition}\label{ss:3.4} \begin{num}\item  For $\alpha\in \Phi'$, $s_\alpha$ restricts to a permutation  $s_\alpha'$  of   $\Phi'$.  \item Define the map $F'\colon \Phi'\rightarrow \text{\rm Sym}(\Phi')$ by $\alpha\mapsto s_\alpha'$. Then $(\Phi',F')$ is a quasi-root system with
  $\Phi'_+$ as a positive system.
  \item The pair $(W',S')$ is a Coxeter system with reflection cocycle $N'\colon W'\rightarrow \mc{P}(T')$
given by 
$N'(w)=\mset{s_\alpha\mid \alpha\in \Psi_+\cap w(-\Psi'_+)}=\mset{s_\alpha\mid \alpha\in \Phi'_+\cap w(-\Phi'_+)}$.
\item Restriction induces an isomorphism $W'=\mpair{s_\alpha\mid \alpha\in \Phi'}\xrightarrow \cong 
\mpair{s'_\alpha\mid \alpha\in \Phi'}$.
\item $\Pi'$ is the set of all simple roots for $\Phi'_+$ in $\Phi'$. \item Let $(\widehat {T'},F'')$ denote the standard abstract root system of $(W',S')$.
 There is a $W'$-equivariant surjection $\Phi'\rightarrow \widehat {T'}$ given by
$\alpha\mapsto (s_\alpha,\epsilon)$ for $\alpha\in \epsilon \Phi'_+$ where $\epsilon\in \set{\pm 1}$.
If $\set{\pm 1}$ acts simply transitively on the fibers of the reflection map of  $(\Phi', F')$, this surjection is an isomorphism $(\Phi',F')\rightarrow (\widehat {T'},F'')$ of quasi-root systems.
  \end{num}\end{proposition}
  \begin{proof}  Part (a) follows from the definition if $\alpha\in \Pi$ and then from \eqref{eq:1} in general.
  Part (b) follows readily from the corresponding facts for $\Psi$.  
  Define the function  $ N_{\Psi_+}\colon W\rightarrow \mc{P}(T)$ by
  $N_{\Psi_+}(w):=\mset{s_\alpha\mid \alpha\in \Psi_+\cap w(-\Psi_+)}$.  
  By \ref{ss:2.4}(a), $N_{\Psi_+}$ is a cocycle, and by definition of $S'$,
  $N_{\Psi_+}(s)=\set{s}$ for $s\in S'$.  The cocycle condition implies $N_{\Psi_+}(w)\subseteq T'$ for all
  $w\in W'$, so 
  $N_{\Psi_+}$ restricts to a cocycle $N'\colon W'\rightarrow\mc{P}(T')$,
 satisfying $N(s)=\mset{s}$ for all $s\in S'$. By \ref{ss:3.1}, $(W',S')$ is  a Coxeter system with reflection cocycle $N'$ where  $N'(w)=\mset{s_\alpha\mid \alpha\in \Psi_+\cap w(-\Psi_+)}$. Now to prove (c), it will suffice to show that for $w\in W'$,  $\Psi_+\cap w(-\Psi_+)=\Phi'_+\cap w(-\Phi_+)$. Clearly, the right hand side is included in the left. But if $\alpha\in \Psi_+$ with $w^{-1}(\alpha)\in -\Psi_+$, then
 from above $s_{\alpha}\in T'$, so $\alpha\in \Phi'_+$ and $w^{-1}(\alpha)\in -\Phi'_+$, which proves we have equality.
 Now for (d), suppose that $w\in W'$ has restriction $w_{\vert \Phi'}=\Id_{\Phi'}$. 
 Then $\Phi_+'\cap w(-\Phi'_+)=\Phi_+'\cap w_{\vert \Phi'}(-\Phi'_+)=\emptyset$ so
$ N(w)=\emptyset$ and $w=\Id_{\Psi}$ by \ref{ss:3.1}.
 For (e), it is clear that $\Pi'$ is a subset of the set of all simple
roots for $\Phi'_+$ in $\Phi'$. On the other hand,  if $\alpha\in \Phi'$ is a simple root
for $\Phi'_+$ in $\Phi'$, then  the definitions readily give that $N'(s_\alpha)=\set{s_\alpha}$, and so $s_\alpha\in S$ which implies $\alpha\in \Pi$. Finally,  (f) follows from \ref{ss:2.4}(d). 
   \end{proof}
   
   \begin{definition}\label{ss:3.5a} A quasi-positive system $\Psi_+$ for   a quasi-root system $(\Psi,F)$  will be said to be  generative  if for some set $S'$ of simple reflections for $\Psi_+$, one has 
   $\Phi=\Psi$ (or equivalently, $W'=W$) in \ref{ss:3.3} above. By \ref{ss:mingen},  $S'$ is necessarily the set of all simple reflections for $\Psi_+$.\end{definition}
   We then  call $(W,S')$ the Coxeter system attached to  $\Psi_+$.
  As we observe later, the isomorphism type of $(W,S')$ as Coxeter system in general depends on the choice of   $\Psi_+$.  The sets of quasi-positive systems and of generative quasi-positive systems
   are  clearly both stable under  the $W$-action.

 \begin{example} \label{ss:B2exmp}Let $(W,S)$ be a Coxeter system of type $B_2$ with $S:=\set{r,s}$ as its set of simple reflections (so $rs$ has order $4$). The set of reflections is $T=\set{r,rsr,srs,s}$. We consider the standard abstract root system $\widehat{T}=T\times\set{\pm 1}$ of $(T,S)$, which has standard positive system $\widehat{T}_+:=T\times\set{1}$.
 One can readily check that $\Psi_+:=\set{(s,1),(srs,1),(r,1), (rsr,-1)}$ is a generative  quasi-positive system for $\widehat{T}$, with $\Delta:=\set{(s,1),(srs,1)}$ as corresponding set of abstract simple roots and $S':=\set{s,srs}$ as the corresponding set of abstract simple reflections. 
 Observe that $\widehat{T}_+$  and $\Psi_+$ are not $W$-conjugate. \end{example}  
 
\begin{definition} \label{ss:3.5b}  An abstract root system is a  quasi-root system $(\Psi,F)$ for which there exists some generative quasi-positive system. Elements of $\Psi$ will then be called roots instead of quasi-roots.
\end{definition} 
   Note the pair $(W,T)$ where $T=\mset{s_\alpha\mid \alpha\in \Psi}$ depends only on $(\Psi,F)$ and not on the choice of generative quasi-positive system.

   \subsection{}  Fix a Coxeter system $(W,S)$.    
   For any $J\subseteq S$, let  $W_J$ denote
the  subgroup of $W$ generated by $J$. The subgroups $W_J$ (resp., their $W$-conjugates)  are called  standard parabolic subgroups
(resp.,  parabolic subgroups) of $(W,S)$.
A reflection subgroup of $W$ is a subgroup $W'$ which is generated by $W'\cap T$.
 A dihedral reflection subgroup is a reflection subgroup which may be generated by two distinct reflections of  $(W,S)$.

\subsection{}\label{ss:3.7} Here we recall some properties of reflection subgroups; for proofs, see \cite{Dy2} or \cite{Dy0}.    For any  reflection subgroup
 $W' =\pair{W' \cap T}$ of $W$,
\[S' =\chi(W' )=\chi_{(W,S)}(W' ):=
\mset{t\in T\mid N(t)\cap W' =\set{t}}\tag{1}\] is
a set of Coxeter generators for $W' $.
 The corresponding set of reflections and
reflection cocycle for $(W' ,S' )$ are $T':=W' \cap T$ 
and $N' \colon w\rightarrow N(w)\cap
W' $ respectively.  If $R\subseteq T$ is any set of reflections of $W$ with $W' =\mpair{R}$,
then $\cup_{w\in W' } wRw^{-1}= T'$ and $\vert S'\vert\leq \vert R\vert$. In particular, if $R$ is a set of Coxeter generators for $W$ with $R\subseteq T$, then $T=\cup_{w\in W}wRw^{-1}$ and $\vert R\vert=\vert  S\vert$. It is known that any dihedral reflection subgroup is contained in a unique maximal (under inclusion) dihedral reflection subgroup.

The following result follows directly from the above facts and Proposition  \ref{ss:3.1}.
\begin{proposition} \label{ss:3.8} Let $\widehat{T'}=T'\times\set{\pm 1}$ and for $\alpha\in \widehat {T'}$, let $s_\alpha'$ denote the restriction of $s_\alpha$ to a permutation of $\widehat{T'}$.
Let $F'\colon  \widehat{T'}\rightarrow \text{\rm Sym}(\widehat{T'})$ denote the map $\alpha\mapsto s_\alpha'$. Then $(\widehat{T'},F')$ is equal to the standard abstract root system of $(W',S')$.\end{proposition}
\begin{proposition}\label{ss:3.8a} Let $\Delta\subseteq \widehat T_+$.  Then $\Delta$ is the standard set of simple roots of 
the standard  abstract root system ${T'}\times \set{\pm 1}$ of some reflection subgroup $W'$ of $W$, with $T'=T\cap W'$, iff for each $\alpha\neq \beta$ in $\Delta$, $\set{\alpha,\beta}$ is the set of simple roots of the standard positive system of the standard  abstract root system $(\mpair{s_\alpha,s_\beta}\cap T)\times \set{\pm 1}$ of $\mpair{s_\alpha,s_\beta}$. \end{proposition}
\begin{proof} This  is  a direct translation  using \ref{ss:3.1} of  \cite[3.5]{Dy2}.\end{proof}
\subsection{} \label{ss:3.9} If $W'$ is a reflection subgroup of $(W,S)$ and $w\in W$, there is a unique
$y\in W'w$ with $N(y^{-1})\cap W'=\emptyset$. Then $\chi(wW'w^{-1})=y\chi(W')y^{-1}$, by \cite[Lemma 1]{Dy3}. 
As in \text{\rm Proposition \ref{ss:3.7}}, regard the (underlying sets of)  abstract root systems attached to $(W',\chi(W'))$ and $(wW'w^{-1},\chi(wW'w^{-1}))$ 
as subsets of $\widehat T$. 
\begin{proposition} \label{ss:3.10} For $W'$, $w$, $y$ as in $\text{\rm \ref{ss:3.9}}$,   the map  $\alpha\mapsto y(\alpha)\colon \widehat T\rightarrow \widehat T$ 
restricts to an isomorphism from the standard abstract root system of $(W',\chi(W'))$   to the standard abstract root system of $(wW'w^{-1},\chi(wW'w^{-1}))$. This bijection restricts to  a bijection between corresponding sets of abstract positive roots (resp., abstract simple roots). \end{proposition} 
\begin{proof} This follows easily using  \ref{ss:3.1} and the definitions.\end{proof}

We record the following Lemma for use in Section 3.

   \begin{lemma} \label{ss:conj}  Let $J\subseteq S$ such that $(W_J,J)$ is an infinite irreducible Coxeter system.
    If $w\in W$ with $l(wr)>l(w)$ for all $r\in J$
    and $wJw^{-1}\subseteq S$, then $w\in W_K$ where $K:=\mset{r\in S\setminus J\mid rs=sr  \text{ \rm  for all $s\in J$}}$. \end{lemma}
    \begin{proof} This follows directly from a result of Deodhar's 
  (see   \cite[Prop 2.3]{BrHow})  on conjugacy of parabolic subgroups and is even implicit as a very special case of the discussion after the statement of Theorem A in loc cit.    To prove it explicitly, we use induction on $l(w)$.
    If $w \neq 1_W$, choose $a\in S\cap N(w^{-1})$. From Proposition 2.3 of loc cit, $a\not\in J$ and
  the irreducible  component of $(W_{J\cup\set{a}},J\cup\set{a})$ containing 
    $a$ is a finite Coxeter system (i.e. in the notation of loc cit,   the element  $v[a,J]$ is defined). Since $W_J$ is infinite and irreducible,
     this implies that this irreducible component has $a$ as its only simple reflection and hence that $a\in K$.
      Now  $l(war)>l(wa)$ for all $r\in J$, and $(wa)J(wa)^{-1}=wJw^{-1}\subseteq S$ so by induction,
      $wa\in W_K$ and the proof is complete. 
       \end{proof}
  
 \subsection{} \label{ss:3.11} We next  examine more closely  the  result \ref{ss:3.4}
 in the special  case that  $(\Psi,F)$ is the standard abstract root system $(\widehat T, F)$ of a Coxeter system
 $(W,S)$.  For  a quasi-positive system $\Psi_+\subseteq \widehat T$, let $\Delta_{\Psi_+}$ and $S_{\Psi_+}$ denote the sets of simple roots and simple reflections for $\Psi_+$  i.e. \begin{gather}\Delta_{\Psi_+}:=\mset{\alpha\in \Psi_+ \mid s_\alpha (\Psi_+\setminus \set{\alpha})\subseteq \Psi_+}\\ S_{\Psi_+}:=\mset{s_\alpha\mid \alpha\in \Delta_{\Psi_+}}.\end{gather} Let  $W_{\Psi_+}:=\mpair{S_{\Psi_+}}$ denote the subgroup of $W$
  generated by $S_{\Psi_+}$. (The  fact  in \ref{ss:3.4}(c) that $(W_{\Psi_+},S_{\Psi_+})$ is a Coxeter 
  system also   follows from  \cite[1.8]{Dy1} using  \ref{ss:3.1}.) 
  For example, if ${\Psi_+}=\widehat T_+$, then $\Delta_{\Psi_+}=S\times \set{1}$ and $S_{\Psi_+}=S$, as is well known. 

 \begin{definition}  A subset $S'$ of  $T$ (resp., a subset $\Delta$ of $\widehat{T}$) is an abstract set of simple reflections (resp., an abstract set of simple roots) for $\widehat{T}$  if $S'=S_{\Psi_+}$ (resp., $\Delta=\Delta_{\Psi_+}$) for some generative quasi-positive system $\Psi_+$.\end{definition} 
  Note that any abstract set of simple reflections of $\widehat T$ is a set of Coxeter generators of $W$ contained in the set of  reflections of $(W,S)$, but that the converse does not hold (see Example \ref{ss:3.19}).
  
 \begin{lemma}  The map $\Psi_+\mapsto \Delta_{\Psi_+}$ 
 is a  a bijection between the set of generative quasi-positive systems and  the set of abstract sets of simple roots of 
 $(\widehat T,F)$.\end{lemma} 
  \begin{proof}
   Clearly  $\Psi_+$  determines $\Delta_{\Psi_+}$, for any quasi-positive system $\Psi_+$.
 On the other hand, if $\Psi_+$ is a generative quasi-positive system, then by \ref{ss:3.1}, $S_{\Psi_+}=\mset{s_\alpha\mid \alpha\in \Delta_{\Psi_+}}$, $W=\pair{S_{\Psi_+} }$
 and $\Psi_+=\mset{w(\alpha)\mid w\in W, \alpha\in \Delta_{\Psi_+},  l_{\Psi_+}(ws_\alpha)>l(w)}$
 where $l_{\Psi_+}$ is the standard length function of $(W,S_{\Psi_+})$. This shows that $\Delta_{\Psi_+}$ determines $\Psi_+$.\end{proof}

 The following  Lemma will be crucial in Section 3. 
  \begin{lemma} \label{ss:3.13} Let ${\Psi_+}$ be a generative quasi-positive  system for the standard abstract root system $(\widehat{T},F)$ of $(W,S)$, and $W'$ be a reflection subgroup of $(W,S)$. Set $S'=\chi(W')$. Regard the abstract root system  $\widehat{T'}$ of $(W',S')$
 as a subset of $\widehat{T}$ as in $\text{\rm \ref{ss:3.8}}$. Then $\Psi_+':={\Psi_+}\cap \widehat{T'}$ is a generative quasi-positive  system for $\widehat{T'}$, and the set of simple reflections of  $\Psi_+'$ is the set  $\chi_{(W,S_{\Psi_+})}(W')$of canonical generators of $W'$ with respect to the Coxeter system $(W,S_{\Psi_+})$.\end{lemma} 
 \begin{proof} Clearly $\Psi'_+$ is  a quasi-positive system
  for $\widehat{T'}_+$. By \ref{ss:mingen}, it suffices to show 
  that if $\alpha\in \Psi'_+$ with
  $s_\alpha\in  \chi_{(W,S_{\Psi_+})}(W')$, then $\alpha\in \Pi'$ where $\Pi'$ is the set of simple roots for $\Psi'_+$
   in $\widehat{T'}$ i.e. $s_\alpha(\Psi'_+\setminus\set{\alpha})\subseteq \Psi'_+$.
     Now  by \ref{ss:3.7}, $N_{\Psi_+}(s_\alpha)\cap W'=\set{s_\alpha}$ i.e. if $\gamma\in \Psi_+$ with $s_\gamma\in W'$ and $s_\alpha(\gamma)\in -\Psi_+$, then $\gamma=\alpha$.
     Since for $\gamma\in \Psi_+$ we have $s_\gamma\in W'$ iff $\gamma\in \Psi'_+$,
     this gives the desired conclusion.      \end{proof}

    \begin{example} \label{ss:3.14} Let $\set{C_i}_{i\in I}$ be the conjugacy classes of reflections in $W$.
  For each $i\in I$, choose $\epsilon_ i\in \set{\pm 1}$. Then 
  ${\Phi_+}:=\cup_{i\in I} \,(C_i\times\set{\epsilon_i})$ is a quasi-positive system.
  It is easy to check that $S\subseteq S_{\Phi_+}$, so $W_{\Phi_+}=W$ and $S_{\Phi_+}=S$
  by \ref{ss:mingen}.
  The non-standard generative-quasi-positive system $\Psi_+$  for type $B_2$ in Example \ref{ss:B2exmp} arises by conjugation  of one of the quasi-positive systems arising as above.
 \end{example}
  
   Proposition \ref{ss:3.15} below  shows in particular that any generative quasi-positive system with $S$ as the corresponding set of simple reflections arises as in the preceding example. For the proof, we  use the following  fact about conjugacy of simple reflections, which follows readily from the exchange condition and is a  special case of  more general results of Deodhar and Brink-Howlett (see \cite{BrHow}) on conjugacy of parabolic subgroups.

 \begin{lemma}\label{ss:3.14a} Let $(W,S)$ be a Coxeter system and $r,s\in S$, $w\in W$ satisfy
 $wrw^{-1}=s$ and $l(wr)=l(w)+1$. Then there exist $k\in \Nat$ and   sequences  $w_1,\ldots, w_k$ in $W$, $J_1,\ldots, J_k\subseteq S$, $a_0,a_1\ldots, a_k$ in $W$ with the following properties:
 \begin{conds}
 \item $a_0=a$, $a_k=b$
 \item $a_{i-1},a_{i}\in J_i$ for $i=1,\ldots,k$
 \item $\vert J_i\vert=2$ for $i=1,\ldots, k$.
 \item $w_i\in W_{J_i}$, $w_ia_i=a_{i+1}w_i$ and $l(w_ia_i)=l(w_i)+1$ for $i=1,\ldots, k$.
 \item $w=w_k\cdots w_2w_1$ with $l(w)=l(w_1)+l(w_2)+\ldots+ l(w_k)$.
 \end{conds}
  \end{lemma}

  \begin{proposition} \label{ss:3.15} Let $S'\subseteq T$ such that
  $(W,S')$ is a Coxeter system. 
  Consider a family of roots $\Pi'=\set{\alpha_r}_{r\in S'}\subseteq \widehat T$ such that $s_{\alpha_r}=r$ for all $r\in S'$.
  Then $\Pi'$ is the set of simple roots of some  (generative) quasi-positive system $\Psi_+$ for $\widehat{T}$ iff $(rt)^m(\alpha_r)=\alpha_t$  whenever $r\neq t$ are in $S'$ with the order of $rt$  an odd integer   $2m+1$. 
  \end{proposition}
  \begin{proof}   The ``only if'' direction is clear, since under the above conditions on $r$ and $t$,  if we write $\beta=(rt)^m(\alpha_r)$, we have $s_\beta=(rt)^mr(rt)^{-m}=t$ so $\beta\in \pm\set{\alpha_t}$, but
$\beta\in \Psi_+$ since $l'((rt)^mr)>l'((rt)^m)$ where $l'$ is the standard length function of $(W,S')$. 

 For the ``if'' direction,
we shall  show first   that for any $x,y\in W$, $r,t\in S'$ with $xrx^{-1}=yty^{-1}$, $l'(xr)>l'(x)$ and $l'(yt)>l'(y)$,  one has $x(\alpha_r)=y(\alpha_t)$. In fact, using \ref{ss:3.1}, the conditions imply that
$zrz^{-1}=t$  and $l'(zr)>l'(z)$ with $z=y^{-1}x$.  By Lemma \ref{ss:3.14a},  the proof that 
$z(\alpha_r)=\alpha_t$ under these conditions  reduces to its special case in which $z\neq 1$ and $z,r,t$ all lie in some 
finite rank two standard parabolic subgroup of $(W,S')$, say of order $2m$. If $m$ is even,  $z(\alpha_r)=\alpha_t$ is automatic, whereas if $m$ is odd,
 $z(\alpha_r)=\alpha_t$  by the assumptions. 

Next we set $\Psi_+:=\mset{x(\alpha_r)\mid x\in W, r\in S', l'(xr)>l'(x)}$. The above implies that
$\Psi_+$ is a quasi-positive system for $\widehat {T}$.  Using  the exchange condition for $(W,S')$, one sees that $\Pi'$ is contained  in the set of  simple roots of $\Psi_+$, and equality follows by \ref{ss:mingen}.  
   \end{proof}
  \subsection{} \label{ss:3.16} We recall the notion of diagram twisting of Coxeter systems, as introduced in \cite{Mue}.  
  
   Let $(W,S)$ be a Coxeter system with Coxeter matrix $(m_{r,s})_{r,s\in S}$ i.e. for  for $r,s\in S$, let $m_{r,s}\in \Nat_{\geq 1}\cup\set{\infty}$ denote the order of the product $rs\in W$. Suppose that
  $S$ is the disjoint union  $S=J\cup K\cup L\cup M$ where $W_K$ is finite with longest element $w_K$, 
  $M\subseteq\mset{s\in S\mid m_{s,r}=\infty\text{ \rm for all $r\in J$}}$ and $m_{r,s}=2$ for all $r\in L$, $s\in K$.
  We let $J'=w_KJw_K$ and $S':=J'\cup K\cup L\cup M\subseteq T$. Then (\cite{Mue})  $(W,S')$ is a  Coxeter system,
  said to be obtained from $(W,S)$ by a diagram twist.  The Coxeter matrix $(m'_{r,s})_{r,s\in S'}$ of $(W,S')$ is given as follows. If $r,s\in K\cup L\cup M$, then $m'_{r,s}=m_{r,s}$. If $r\in J'$ and $s\in K\cup L\cup M$, then $m'_{r,s}=m'_{s,r}=m_{w_Krw_K,s}$. Finally, if $r,s\in J'$ then $m'_{r,s}=m_{w_Krw_K,w_Ksw_K}$.
 
  \begin{corollary}\label{ss:3.18} Suppose in $\text{\rm \ref{ss:3.15}}$, that $(W,S')$  is obtained by
  twisting $(W,S)$ with notation as in $\text{\rm \ref{ss:3.16}}$.
   Set $\alpha_r=(r,\epsilon_r)$ with $\epsilon _r\in \set{\pm 1}$ 
  for all $r\in S'$. Then $\Pi':=\mset{\alpha_r}_{r\in S'}$ is the set of simple roots of some generative
  quasi-positive system for $\widehat{T}$ iff the following conditions $\text{\rm  (i)---(iii)}$ hold:
  \begin{conds} \item $\epsilon_r=\epsilon_s$ whenever
  $r\neq s$ are in $K\cup L\cup M$ and $m'_{r,s}$ is finite and odd
  \item $\epsilon_r=-\epsilon _s$ when $r\in J'$ and $s\in K$ with
  $m'_{r,s}$ finite and odd
  \item $\epsilon_r=\epsilon_s$ whenever
  $r\in J'$, $s\in L$   and $m'_{r,s}$  finite and odd.\end{conds}
 \end{corollary}
 \begin{proof} The Corollary follows readily on applying  \ref{ss:3.15} to both $S'$ and $S\times\set{1}\subseteq \widehat T$ on noting that each pair of distinct
 elements of $S'$ which are not both in $S$ and  whose product has finite order is conjugated to a  pair
 of elements of $S$ by $w_K$.  \end{proof}

\begin{example}\label{ss:3.19}  Consider the Coxeter system $(W,S)$ with $S=\set{r,s,t,u}$ and  Coxeter graph as at left below.
  Twisting 
  as in \ref{ss:3.16} using  $J=\set{r}$, $K=\set{s,t}$, $M=\set{u}$ and $L=\emptyset$ gives the 
   Coxeter system $(W,S')$, where $S'=\set {r',s,t,u}$ with $r'=stsrsts$ and  with Coxeter graph as at right below:
   
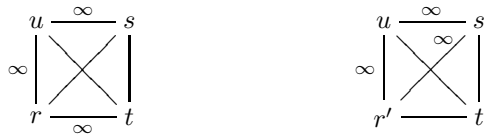
\begin{equation*}
   \xymatrix{{u}\ar@{-}[r]^\infty\ar@{-}[dr]\ar@{-}[d]_\infty&{s}\ar@{-}[d]&&&
   {u}\ar@{-}[r]^\infty\ar@{-}[dr]\ar@{-}[d]_\infty&{s}\ar@{-}[d]&\\
   {r}\ar@{-}[ur]\ar@{-}[r]_\infty&{t}&&&{r'}\ar@{}@<-.8ex>[ur]^(.75
   )\infty\ar@{-}[ur]\ar@{-}[r]&{t}}\end{equation*}
    In this case, \ref{ss:3.18} implies  that $\set{(r',-1), (s,1),(t,1),(u,1)}$ is the set of abstract simple roots of some generative quasi-positive system, with $S'$ as the corresponding set of simple reflections. In particular, this shows that the non-isomorphic
   (irreducible, finitely-generated) Coxeter systems
   $(W,S)$ and $(W,S')$ have isomorphic standard abstract root systems.
   
   On the other hand, take $(W,S)$ with $S=\set{a,b,c,d}$ and Coxeter graph as at left below.
   Twisting with $J=\set{d}$, $K=\set{a}$,  $L=\set{b}$ and $M=\set{c}$ gives the isomorphic Coxeter system $(W,S')$ with $S'=\set{a,b,c,d'}$ where $d'=ada$, and Coxeter graph as at right below. 
   
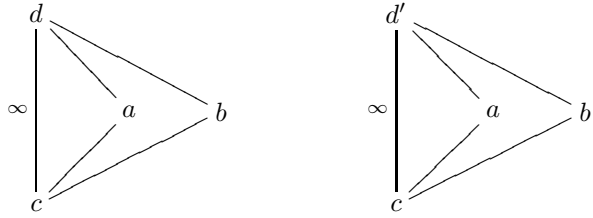
\begin{equation*}
   \xymatrix{{d}\ar@{-}_\infty[dd]\ar@{-}[dr]\ar@{-}[drr]&&&&{d'}\ar@{-}_\infty[dd]\ar@{-}[dr]\ar@{-}[drr]&&\\
   &{a}&{b}&&&{a}&{b}\\
   {c}\ar@{-}[ur]\ar@{-}[urr]&&&&{c}\ar@{-}[ur]\ar@{-}[urr]}\end{equation*}
   Here \ref{ss:3.18} shows there is no generative quasi-positive system with $S'$ as its corresponding set of simple reflections. This shows that arbitrary diagram twists do not necessarily extend to twists of
   the standard abstract root system.\end{example}
  
  \begin{question} \label{ss:3.20} The above suggests the problem  of determining when (finitely generated, irreducible) Coxeter systems have isomorphic standard abstract root systems. 
  One  might also ask whether there is some natural  way of attaching an abstract root system $\widehat \Phi(W,S)$
  to each Coxeter system $(W,S)$ so that two  (say, finitely generated, irreducible) Coxeter systems $(W_i,S_i)$
  are isomorphic iff they have isomorphic abstract root systems $\widehat{\Phi}(W_i,S_i)$.  \end{question} 
  
  In order to limit  the somewhat pathological  behaviour  seen in examples such as  \ref{ss:B2exmp}, \ref{ss:3.14} and \ref{ss:3.19}, we now  introduce combinatorially  
  a notion   which corresponds to that  of betweenness considered in the Introduction (see \ref{ss:5.1}(c)).

\begin{definition} Let $\alpha,\beta,\gamma\in \widehat T$. We say that $\gamma$ is between
$\alpha$ and $\beta$ if one of the following conditions holds:
\begin{conds}\item $\alpha=\pm \beta$ and $\gamma\in \set{\alpha,\beta}$.
\item $\alpha\neq\pm \beta$ and for any maximal dihedral reflection subgroup $W'$ of $(W,S)$, 
all $w\in W'$ and all $\epsilon\in\set{\pm 1}$ with $\epsilon w(\alpha),\epsilon w(\beta)\in \widehat{T'}_+:=(W'\cap T)\times \set{1}$, one has $\epsilon w(\gamma)\in\widehat{T'}_+$.\end{conds}
We let $[\alpha,\beta]$ denote the set of all $\gamma\in \widehat T$ such that $\gamma$ is between $\alpha$ and $\beta$.
\end{definition}

If $\alpha\neq \pm\beta$, the only possible maximal dihedral reflection subgroup
$W'$ as in (ii) is the one containing $s_\alpha$ and $s_\beta$ (though possibly several pairs
$(w,\epsilon)$ may satisfy the conditions of (ii)). On the other hand, using \ref{ss:3.9}--\ref{ss:3.10} and the fact that the set of  maximal
dihedral   reflection subgroups of $W$ is closed under conjugation by elements of $W$, one sees that 
an equivalent condition to (ii) would be obtained by replacing ``$w\in W'$'' in (ii) by  ``$w\in W$.''
From this, it is clear for any $w\in W$, $\epsilon\in \set{\pm1 }$ and $\alpha,\beta,\gamma$ in $\widehat T$, that $\gamma\in[\alpha,\beta]$ iff 
$\epsilon w(\gamma)\in [\epsilon w(\alpha),\epsilon w(\beta)]$. 

\subsection{}\label{ss:3.22} We provide here without proof  a more  concrete description of betweenness (cf \ref{ss:5.1}(c)). Let $\alpha,\beta\in \widehat{T}$
with $\alpha\neq \pm \beta$. Let $W'$ be the maximal dihedral reflection subgroup $W'$
of $W$ containing $\set{s_\alpha,s_\beta}$, and set  $T'=W'\cap T$. Let $ \chi(W')=\set{r,s}$ and $m=\card{T'}$.
Consider a  diagram of $m$ straight lines through  the origin $\bullet=(0,0)$ in the  plane $\real^2$:
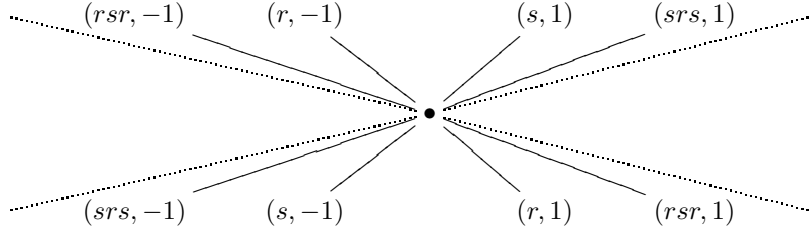
\begin{equation*}\xymatrix{{}\ar@{.}[rrrd]&{(rsr,-1)}\ar@{-}[rrd]&{(r,-1)}\ar@{-}[dr]&&(s,1)\ar@{-}[dl]&(srs,1)\ar@{-}[dll]&{}\ar@{.}[dlll]\\
&&&{\bullet}&&&&\\
{}\ar@{.}[rrru]&{(srs,-1)}\ar@{-}[rru]&{(s,-1)}\ar@{-}[ur]&&{(r,1)}\ar@{-}[ul]
&{(rsr,1)}\ar@{-}[ull]&{}\ar@{.}[ulll]}\end{equation*}
If $m$ is finite the lines
are supposed to pass, say,  through the vertices of some regular $k$-gon  with centroid at $\bullet$, where $k=m$ if $m$ is odd and $k=2m$ if $m$ is even, while if $m$ is infinite, the lines  are taken as those passing through, say,   the points
$(n,\pm 1)$  for $n\in \Int\setminus\set{0}$.  
The resulting  $2m$ (closed) rays with $\bullet$ as endpoint are  labelled by the elements of $\widehat{T'}\times \set{\pm 1}$ as suggested by the diagram, so that each line is the union of rays with labels $(t,1)$ and $(t,-1)$ for some (unique) $t\in T$.

Now for  $\gamma\in \widehat T$, one has $\gamma\in [\alpha,\beta]$ iff the ray labelled $\gamma$ is 
in the convex closure of  the union of the  rays labelled $\alpha$ and $\beta$.  (This is easy to see either directly from the definition , or using \ref{ss:5.1}(c)). 

\begin{lemma}\label{ss:det} For any  $\alpha,\beta\in \widehat T$, $s_\beta(\alpha)$ is determined completely by $\alpha$, $\beta$,  the function $(\gamma,\delta)\mapsto[\gamma,\delta]\colon \widehat T\times \widehat{T}\rightarrow \mc{P}(\widehat{T})$,   and the action of $\set{\pm 1}$ on $\widehat T$.
\end{lemma}\begin{proof} First,  if $\beta=\pm \alpha$, then $s_\beta(\alpha)=- \alpha$, so we suppose henceforward that $\beta\neq \pm\alpha$.
Consider the  smallest (under inclusion)  subset $\widehat{T'}$  of $\widehat {T}$ containing $\set{\alpha,\beta}$, closed under the action of $\set{\pm 1}$ and 
such that if $\gamma,\delta\in \widehat{T'}$, then $[\gamma,\delta]\subseteq \widehat{T'}$. It is easy to see that $\widehat{T'}= T'\times\set{\pm 1}$ where $W'$ is the maximal dihedral reflection subgroup of $W$
 containing $\mset{s_\alpha,s_\beta}$ and $T'=W'\cap T$. 
 
 Now there is a unique  element $\gamma\in \widehat{T'}$ such that 
 $\alpha\in[\beta,\gamma]$ and $[\gamma,-\beta]=\set{\gamma,-\beta}$ (in fact, 
 $\set{\beta,\gamma}$ is the set of simple roots of some quasi-positive system $\epsilon w'(T'\times\set{1})$, for some $\epsilon\in \set{\pm 1}$ and $w'\in W'$, of the standard root system $(W'\cap T)\times\set{\pm 1}$ of $(W',\chi(W'))$;  further, $ [\beta,\gamma]=\epsilon w'(T'\times\set{1})$). Define the cardinalities $m=\vert [\beta,\alpha]\vert $ and $n=\vert [\alpha,\gamma]\vert$, at most one of which is infinite. 
Then one can easily check that  $\delta:=s_\beta(\alpha)$ is determined case by case  as follows. If $m=n+1$, then $\delta=\beta$.  If $m<n+1$ (resp., $m>n+1$), then 
$\delta$ is the unique element of $[\beta,\gamma]$ such that $\vert[\gamma,\delta]\vert=m-1$
(resp.,  $\vert[\beta,\delta]\vert=n+1$).\end{proof}

\begin{definition}A subset $P$ of $\widehat T$ is said to be closed (in $\widehat T$) if for any  $\alpha,\beta\in P$ and $\gamma\in \widehat T$
such that $\gamma$ is between $\alpha$ and $\beta$, one has  $\gamma\in P$.
One says that $P$ is biclosed (in $\widehat T$) if $P$ and $\widehat T\setminus P$ are both closed in $\widehat T$.
Similarly, for a reflection subgroup $W'$ of $W$, we may say that a subset
$P$ of $\widehat{T'}$ is closed or biclosed in $\widehat {T'}=T'\times\set{\pm 1}$.\end{definition}
There should be no confusion with the notion of a closed subset of the root system of a finite Weyl group as defined in \cite{Bou},  which will not be used in this paper.

\begin{example}  \label{ss:3.25} Suppose that $(W,S)$ is a dihedral Coxeter system, say $S=\set{r,s}$ where $r\neq s$. Let $B$  denote the set of all biclosed quasi-positive systems $\Psi_+$ for $\widehat T$.
It is easy to see that if $W$ is finite, then $B=\mset{w(\widehat T_+)\mid w\in W}$.  If $W$ is infinite, 
$B=\mset{\epsilon w(\widehat T_+)\mid w\in W, \epsilon\in \set{\pm 1}}\cup\set { 
\Psi_+, \widehat T\setminus \Psi_+}$ where \[\Psi_+=\mset{(t,1)\mid t\in T,  r\in N(t)}\cup\mset{(t,-1)\mid t\in T, s\in N(t)}.\]
Hence   if  $P\in B$ is generative,  then    $P$ is $W$-conjugate to $\pm \widehat T_+$.
\end{example}

 \begin{definition} A subset ${\Psi_+}$ of $\widehat T$ is an abstract   positive system for $(W,S)$
 if ${\Psi_+}$ is a biclosed,  generative, quasi-positive system of $\widehat T$. Then $\Delta_{\Psi_+}$ is called the corresponding abstract root basis for ${\Psi_+}$, and $S_{\Psi_+}$ is called
 the set of simple reflections of $W$ corresponding to ${\Psi_+}$. Recall from \ref{ss:3.11} that $(W,S_{\Psi_+})$ is a Coxeter system.
 \end{definition}
 
 When confusion with the notion of a positive system of roots of a real root system (as in 
 Section 2) seems unlikely,  we may call abstract positive systems just positive systems. 
 
 \subsection{}\label{ss:3.27}  If ${\Psi_+}$ is a  generative quasi-positive  system for $\widehat{T}$ for  $(W,S)$, then there is a natural identification (as quasi-root systems)  of $\widehat T$ with the standard abstract root system of $(W,S_{\Psi_+})$, such that ${\Psi_+}$ identifies with 
 the standard positive system of the latter abstract root system.   Moreover, under this identification,
 $\widehat T_+$ corresponds to a   generative quasi-positive system of the standard abstract root system of $(W,S_{\Psi_+})$.
 
 The above paragraph remains true with ``generative quasi-positive system'' replaced by ``positive system'' everywhere. 
 Moreover, in that case,  the above  identification  of abstract root systems preserves the notion of betweenness.

 \subsection{} \label{ss:reduc} Let $(W,S)$ be a Coxeter system with irreducible components $(W_i,S_i)$ for $i\in I$. 
Then the standard abstract root system $(\widehat T,F)$ of $(W,S)$ identifies (as quasi-root system) with
the union as in \ref{ss:2.1a} of the (quasi-root systems arising as) standard abstract root systems $(\widehat{T_i},F_i)$ of $(W_i,S_i)$, for $i\in I$. Under this identification, the standard root basis (resp., standard set of positive roots) of $(\widehat T,F)$ is the union of the standard  root bases 
(resp., standard sets of positive roots) of $(\widehat{T_i},F_i)$.

\begin{lemma} For $\Psi_+\subseteq \widehat T$, $\Psi_+$ is a  generative quasi-positive system
(resp.,  positive system)
of $\widehat T$ iff for each $i$, $\Psi^i_+:=\Psi_+\cap \widehat T_i$ is a generative quasi-positive system
(resp., positive system) of $\widehat{T_i}$. In that case, the corresponding set of simple roots
(resp., abstract root basis) $\Delta_{\Psi_+}$ of $\Psi_+$ is the disjoint union $\cup_i  \Delta_+^i$
where $\Delta^i_+$ is the set of simple roots  (resp., abstract root basis) of $\Psi^i_+$ in $\widehat{T_i}$.\end{lemma}
\begin{proof} Observe that if $t\in W_i\cap T$ and $s\in W_j\cap T$ with $i\neq j$  are  reflections of $W$ which are contained in  distinct irreducible components  of $W$, then $st=ts$ and the maximal dihedral reflection subgroup containing $s$ and $t$ is just $\set{1,s,t,st}$.  Using this, the  lemma is easily proved from the definitions. We omit further details of the proof. \end{proof}

   \begin{lemma}\label{ss:3.28} Let $W'$ be a reflection subgroup of $(W,S)$. Set $S'=\chi(W')$. Regard the abstract root system  $\widehat{T'}$ of $(W',S')$
 as a subset of $\widehat{T}$.  Fix $\alpha,\beta,\gamma\in \widehat{T'}$ and $P\subseteq \widehat T$. \begin{num}
 \item $\gamma$ is between $\alpha$ and $\beta$ in $\widehat{T'}$ iff 
  $\gamma$ is between $\alpha$ and $\beta$ in $\widehat{T}$.
  \item  If $P$  is closed (resp., biclosed) in $\widehat T$, then $P\cap \widehat{T'}$ is closed (resp., biclosed) in $\widehat{T'}$.
 \item If $\Psi_+$ is a positive system for $\widehat T$, then $\Psi_+\cap \widehat{T'}$ is a positive system for $\widehat{T'}$.
 \item Suppose that $W'$ is a dihedral reflection subgroup.
 If $\Psi_+$ is a positive system for $\widehat T$, then
 $\Psi_+\cap W'$ is conjugate in $W'$ to $T'\times \set{\epsilon}$ for some
 $\epsilon\in \set{\pm 1}$, where $T'=W'\cap T$. Further, the abstract root basis of  the abstract root system $\widehat {T'}$ of $W'$ corresponding to
the positive system  $\Psi_+\cap \widehat {T'}$ is then  conjugate by an element of  $W'$ to $\chi(W')\times\set{\epsilon}$.
 \end{num}\end{lemma} 
\begin{proof} Note that every maximal dihedral reflection subgroup of $W'$ is contained in a maximal dihedral reflection subgroup of $W$, and any maximal dihedral reflection subgroup of $W$
 which contains at least two reflections of $W'$ intersects $W'$ in a maximal dihedral reflection subgroup of $W'$.  Now (a) follows these remarks together with \ref{ss:3.8}--\ref{ss:3.10}
 (or by other arguments, e.g. involving \ref{ss:3.22}).   Then (b) follows from (a),
 and (c) follows from (b) and \ref{ss:3.13}. Using (c), (d) reduces to the case in which $W=W'$ is dihedral and  $\chi(W')=S$,
 when it follows using Example \ref{ss:3.25}.
 \end{proof} 
 
 \begin{lemma}\label{ss:3.29} Let $(W_i,S_i)$ for $i=1,2$ be Coxeter systems
with abstract root systems $\Phi(W_i,S_i)$ for $i=1,2$. Let
$\theta\colon \Phi(W_1,S_1)\rightarrow \Phi(W_2,S_2)$ be a bijection.
Consider the following conditions $\text{\rm (i)--(iii)}$:
\begin{conds}
\item for $\alpha,\beta,\gamma\in \Phi(W_1,S_1)$,
$\gamma$ is between $\alpha$ and $\beta$ iff $\theta(\gamma)$ is between 
$\theta(\alpha)$ and $\theta(\beta)$) in $\Phi(W_2,S_2)$.
\item $\theta(\pm \alpha)=\pm \theta(\alpha)$ for all $\alpha\in \Phi_1(W,S)$
\item \label{eq:comm}$\theta(s_\alpha(\beta))=s_{\theta(\alpha)}(\theta(\beta))$ for all $\alpha,\beta\in \Phi(W_1,S_1)$
\end{conds} 
 \begin{num} \item Conditions $\text{\rm (i)}$ and $\text{\rm (ii)}$ above hold iff 
  $\text{\rm (i)}$ and $\text{\rm (iii)}$  hold. In that case,  $\theta$ is an isomorphism of abstract root systems.
 \item  Condition   $\text{\rm (iii)}$  implies that  there is  a group isomorphism $W_1\rightarrow W_2$ mapping $s_\alpha\mapsto s_{\theta(\alpha)}$.
 \end{num}
We say that $\theta$ preserves betweenness if  conditions  $\text{\rm (i)--(iii)}$  above are satisfied.
  \end{lemma} 
  \begin{proof} By taking $\beta=\alpha$,  one sees that  (iii) implies (ii). That (i) and (ii) implies (iii) is a direct consequence of \ref{ss:det}. The remainder of (a) is trivial. 
Part (b) may be proved  using the following  fact, which is well-known and easily checked:
if $(W,S)$ is a Coxeter system with set of reflections $T$,  then $W$ is isomorphic to the group
 generated by generators $\overline{t}$ for $t\in T$ subject to relations $\overline{t}\,\overline{t'}\,\overline{t}=\overline{tt't}$ for all $t,t'\in T$.
   \end{proof}

\begin{theorem} \label{ss:abchar}Let $\Delta\subseteq \widehat T$ and set $S':=\mset{s_\alpha\mid \alpha\in \Delta}$. 
Consider the conditions $\text{\rm (i)--(iii)}$ below: 
\begin{conds}\item $\mpair{S'}=W$. 
 \item for  all $\alpha,\beta\in \Delta$ with $\alpha\neq \beta$, $\set{\alpha,\beta}$ is an abstract root basis for the root system $(W_{\alpha,\beta}\cap T)\times\set{\pm 1}$ of $W_{\alpha,\beta}=\mpair{s_\alpha,s_\beta}$ i.e. for some $w\in W_{\alpha,\beta}$ and 
 $\epsilon\in \set{\pm 1}$, $\set{\alpha,\beta}=w( \chi(W_{\alpha,\beta})\times\set{\epsilon})$.
 \item $\Delta$ is contained in some biclosed quasi-positive system $\Psi_+$  of 
 $\widehat T$.
\end{conds} 
Then:
 \begin{num}\item If  $\text{\rm (i)--(ii)}$  hold,  $\Delta$ is the set of simple reflections of some 
 quasi-positive system of $\widehat{T}$; in particular, $(W,S')$ is a Coxeter system.
 \item The conditions $\text{\rm (i)--(iii)}$ all hold iff $\Delta$ is an abstract root basis of $\Psi$.
 In that case, $\Psi_+$ in $\text{\rm (ii)}$ is the unique positive system with $\Delta$ as its set of simple roots.
 \end{num}
 \end{theorem}
   \begin{proof}   
   First we prove (a). Assume that (i)--(ii) hold.  From (ii), $\Delta\cap -\Delta=\emptyset$. To show $(W,S')$ is a Coxeter system, we shall use a characterization of Coxeter systems in \cite{Gold}.
 Define the (left) $W$-set  $\Omega:=W$ with $W$-action by left translation. We regard  $W$  as a group of permutations of $\widehat{T}$.
 For $\alpha\in \widehat{T}$, define an ``abstract halfspace'' $H_\alpha\subseteq \Omega$ by
 $H_\alpha:=\mset{w\in W\mid \alpha\in w(\Psi_+)}$.  We have $\Omega=H_\alpha\cup H_{-\alpha}$, $H_\alpha\cap H_{-\alpha}=\emptyset$. Also, $1\in \cap_{\alpha\in \Psi_+}H_\alpha$,
 and $w(H_\alpha) =H_{w(\alpha)}$ for all $w\in W$. Further, for $\alpha,\beta\in \Psi$, $H_\alpha\cap H_\beta=\cap_{\gamma\in [\alpha,\beta]}H_\gamma$
 since $w(\Psi_+)$ is closed in $\widehat T$. 
 
 We claim that $X:=\mset{(s_\alpha,H_\alpha)\mid \alpha\in \Delta}$ is a pairwise proper
 system of reflections for $W$ acting on $\Omega$ in the sense of loc cit (but note that we use left group actions instead of right actions as in loc cit). The main result of loc cit is that this implies that $(W,S')$ is a Coxeter system, with $S'=\mset{s_\alpha\mid \alpha\in \Delta}$, and then (a) follows from
 (ii) and \ref{ss:3.15}. According to the definition of system of reflections in loc cit,
  $X$ is a system of reflections since for $\alpha\in \Delta$, 
 $s_\alpha$ is an involution, $H_\alpha\subseteq \Omega$, $H_\alpha\cap s_\alpha(H_\alpha)=\emptyset$, and 
 $H_\Delta:=\cap _{\beta\in \Delta} H_\beta\neq \emptyset$ (the last since $1\in H_\Delta$ by  (ii) and the above). To show $X$ is pairwise proper, fix
 $\alpha\neq \beta$ in $\Delta$. Set $W_{\alpha,\beta}=\mpair{S_{\alpha,\beta}}$
 where $S_{\alpha,\beta}=\set{s_\alpha,s_\beta}$, and let $l_{\alpha,\beta}$ denote the length
 function of the (dihedral) Coxeter system $(W_{\alpha,\beta}, S_{\alpha,\beta})$.
 Set $H_{\alpha,\beta}=H_\alpha\cap H_\beta$, and note it is non-empty.
 According to the definition of pairwise proper system of reflections, we have to show  for all $w\in W_{\alpha,\beta}$, that  either  $w(H_{\alpha,\beta})\subseteq H_\alpha$, or that $w(H_{\alpha,\beta})\subseteq s_\alpha(H_\alpha)$ and $l_{\alpha,\beta}(s_\alpha w)<l_{\alpha,\beta}(w)$.  
 But by (ii), $\mset{\alpha,\beta}$ is an abstract root basis for the abstract root system $\widehat{T_{\alpha,\beta}}=(W_{\alpha,\beta}\cap T)\times\set{\pm 1}$ of  $W_{\alpha,\beta}$. The abstract positive system
 corresponding to $\mset{\alpha,\beta}$ is $\Psi_{\alpha,\beta}=[\alpha,\beta]\cap\widehat{T_{\alpha,\beta}}$.  Hence  by Proposition \ref{ss:3.1},
 for $w\in W$. \begin{equation}\label{eq:5}\mset{\gamma\in \Psi_{\alpha,\beta}\mid l_{\alpha,\beta}(s_\gamma w)<l_{\alpha,\beta}(w)}=\Psi_{\alpha,\beta}\cap w(-\Psi_{\alpha,\beta}).\end{equation}
 Note that from above, $H_{\alpha, \beta}=\cap _{\gamma\in {\Psi_{\alpha,\beta}}}H_\gamma$ and 
 $w(H_{\alpha,\beta})=\cap _{\gamma\in w{\Psi_{\alpha,\beta}}}\,H_\gamma$ . Now
 let $\gamma\in \Psi_{\alpha,\beta}$. If $l_{\alpha,\beta}(s_\gamma w)<l_{\alpha,\beta}(w)$, \eqref{eq:5} gives 
 $-\gamma\in w(\Psi_{\alpha,\beta})$ so $w(H_{\alpha,\beta})\subseteq H_{-\gamma}=s_\gamma(H_\gamma)$. Otherwise, $l_{\alpha,\beta}(s_\gamma w)>l_{\alpha,\beta}(w)$, so \eqref{eq:5} gives 
 $\gamma\in w(\Psi_{\alpha,\beta})$ and $w(H_{\alpha,\beta})\subseteq H_\gamma$. For $\gamma=\alpha$, this is what we had to show to prove that $X$ is pairwise proper.
 
 Next we prove (b), by an argument  independent of (a). The conditions (i)--(iii)  are clearly necessary for $\Delta$ to be an abstract root basis.  Now assume that (i)--(iii) hold.  Let $l'$ be the length function of $(W,S)$. We first claim that
 if $w\in W$ and $\alpha\in \Delta$ with $l'(ws_\alpha)\geq l'(w)$, then $w(\alpha)\in \Psi_+$ (this is an abstract version of a well-known fact \ref{ss:4.3}(b) about  real root systems, and the following proof is essentially the same). The claim is easily checked if $(W,S)$ is dihedral, and we reduce to that case
by induction on $l'(w)$. The claim is trivial if  if $l'(w)=0$.  Otherwise, write $w=w's_\beta$ where $\beta\in \Delta$ and $l'(w')=l(w)-1$.  Necessarily, $\beta\neq \alpha$. Write $w=xy$ where $x\in W$, $y\in \mpair{s_\alpha,s_\beta}$ and $l'(x)$ is minimal. Then $l'(xs_\alpha)\geq l'(x)$, $l'(xs_\beta)\geq l'(x)$ and $l'(x)<l'(w)$. By induction, $x(\alpha)\in \Psi_+$ and $x(\beta)\in \Psi_+$.
 Note that $l''(ys_\alpha)\geq l''(y)$ where $l''$ is the length function of $\mpair{s_\alpha,s_\beta}$ with respect to its Coxeter generators $\set{s_\alpha,s_\beta}$. By the claim for the dihedral case,
 it follows that $y(\alpha)$ is between $\alpha$ and $\beta$ i.e.  $y(\alpha)\in [\alpha,\beta]$.
 Then $w(\alpha)=xy(\alpha)\in [x(\alpha),x(\beta)]$. Since $x(\alpha), x(\beta)\in \Psi_+$ and $\Psi_+$ is closed, $w(\alpha)\in \Psi_+$ as required to prove the claim. From the claim, it follows that
  if $w\in W$ and $\alpha\in \Delta$ with $l'(ws_\alpha)\leq l'(w)$, then $w(\alpha)\in -\Psi_+$.
  It is easy to deduce from this  by standard arguments that $(W,S')$ satisfies the exchange condition and is a Coxeter system (as of course also follows from (a)). Using (ii), Proposition \ref{ss:3.15} and
  its proof imply that $\Psi'_+:=\mset{w(\alpha)\mid w\in W,\alpha\in \Pi, l'(ws_\alpha)\geq l'(w)}$
  is a generative, quasi-positive system with $\Delta$ as its set of simple roots.
  From above, $\Psi'_+\subseteq \Psi_+$. Since $\Psi'_+$ and $\Psi_+$ are both quasi-positive systems, it follows that $\Psi_+=\Psi'_+$. By the definitions,  $\Psi_+$ is a positive system of $\widehat{T}$, and we have seen it has $\Delta$ as its set of simple roots. This completes the proof of (b).
      \end{proof}
   \begin{remark} For $\gamma\in \Psi_+$, one has in the above proof of (a)  above that
   $H_{\epsilon\gamma}=\mset{w\in W\mid \epsilon l'(s_\gamma w)<\epsilon l'(w)}$ for $\epsilon\in \set{\pm 1}$, where $l'$ is the length function of the Coxeter system $(W,S')$. The sets $H_{\gamma}$ for $\gamma\in \widehat T$ are  the sets of chambers of the  ``roots'' of $(W,S')$ as defined in the study of the chamber system attached to $(W,S')$ (see  for example \cite{Bahls}). \end{remark}
   \begin{example} Consider the standard abstract root system
   $\widehat{T}$ of the  Coxeter system $(W,S)$ of type $\tilde{B}_2$ with
   Coxeter graph
   \[ \xymatrix{ {r}\ar@{=}[r]&{s}\ar@{=}[r]&{t.}}\]
   Define the subset $\Delta:=\set{(s,1),(srs,1),(t,1)}$ of $\widehat{T}$  and the corresponding set $S':=\mset{s_\alpha\mid \alpha\in \Delta}=\set{s,srs,t}$ of reflections. Note that $S'$ generates $W$ and that for each distinct  $\alpha,\beta\in \Delta$, $\set{\alpha,\beta}$ is the abstract set of simple roots of some generative quasi-positive system for the standard abstract root system $(\mpair{s_\alpha,s_\beta}\cap T)\times \set{\pm 1}$ of 
   $\mpair{s_\alpha,s_\beta}$. In fact, for $\set{s_\alpha,s_\beta}=\set{s,srs}$ this follows from
   \ref{ss:B2exmp} and for the other pairs, $\mset{\alpha,\beta}=\chi(\mpair{s_\alpha,s_\beta})\times\set{1}$. However, $(W,S')$ is not a Coxeter system, showing the assumptions (ii) in \ref{ss:abchar}(a) can't be  weakened in an obvious way.  If $\vert S\vert$ is finite, the condition (iii) in \ref{ss:abchar}(b) can be replaced by the assumption that $\vert \Delta\vert=\vert S\vert$ (see \ref{ss:abchar2}), but I do not know a proof  of this using  abstract root systems. 
       \end{example}

 \section{Real root  systems} 
   This section describes  the real   reflection representations and
corresponding root systems which we consider in this paper. The results are variants
 of standard facts and can can be proved in a similar way or deduced from the standard versions
(see \cite{Bou}, \cite{Hum}, \cite{Deo}, \cite{Dy2}, \cite{How}), with a few minor differences which we indicate.
\subsection{}  We say a subset $\Pi$ of a real vector space $V$ is positively independent
(resp., strongly positively independent) if $\sum_{\alpha\in \Pi}c_\alpha \alpha=0$ with $c_\alpha\in \real$ almost all non-zero and all $c_\alpha$ non-negative  (resp., and at most one $c_\alpha$ negative) 
implies that all $c_\alpha=0$. Thus, $\Pi$ is positively independent  if it all its elements are non-zero
and for any finite subset of $\Pi$, the set of non-negative linear combination of elements of $\Pi$ is
a pointed polyhedral cone in $V$. Also,  $\Pi$ is  strongly positively independent  if any finite subset of $\Pi$ is a set of representatives of the extreme rays of some pointed polyhedral cone in $V$. 
\subsection{} \label{ss:4.2}
We consider two
$\real$-vector spaces
$V$,
$V'$ with a given  fixed $\real$-bilinear pairing $\mpair{\, ,\,}\colon V\times V'\rightarrow
\real$. For any $\alpha\in V$, $\alpha'\in V$ with
$\mpair{\alpha,\alpha'}=2$,  we let $s_{\alpha,\alpha'}\in \text{\rm GL}(V)$ be
the linear map (pseudoreflection) given by   $v\mapsto v-\pair{v,\alpha'}\alpha$,
and define $s_{\alpha',\alpha}\in \text{\rm GL}(V')$ similarly.
Assume given subsets $\Pi\subseteq V$, $\ck\Pi\subseteq V'$ and a bijection
$\iota\colon \Pi\rightarrow
\ck
\Pi$ denoted $\alpha\mapsto \ck \alpha$ such that $\pair{\alpha,\ck \alpha}=2$ for $\alpha\in \Pi$.
Let $S  :=\mset{s_{\alpha,\ck\alpha}\mid\alpha\in \Pi}$, $W$ denote the subgroup of $\text{\rm GL}(V)$ generated by $S  $,
$\Phi:=\cup_{w\in W}\, w(\Pi)$, and $\Phi_+:=\Phi\cap \sum_{\alpha\in \Pi}\real_{\geq 0}\alpha$.
 Define $S  '$, $W'$, $\ck\Phi$, $\ck\Phi_+$ similarly using $\ck\Pi\subseteq V'$ instead of $\Pi\subseteq V$.

\begin{lemma}\label{ss:4.3}  Define $P:=\mset{4\cos^2\frac{\pi}{m}\mid m\in \Nat_{\geq 2}}\cup[4,\infty)\subseteq\real_{\geq
0}$. Consider the following conditions $\text{\rm (i)--(iii)}$:
\begin{conds}\item $\Phi=\Phi_+\cup(-\Phi_+)$
\item $\ck\Phi=\ck\Phi_+\cup(-\ck\Phi_+)$
\item for $\alpha\neq \beta$ in $\Pi$, we have $\mpair{\alpha,\ck\beta}\leq 0$ and
$c_{\alpha,\beta}:=\mpair{\alpha,\ck\beta}\mpair{\beta,\ck\alpha}\in P$; moreover
$\mpair{\alpha,\ck\beta}=0$ iff $\mpair{\beta,\ck\alpha}=0$  \end{conds}
\begin{num}\item If $\Pi$ is positively independent and $\text{\rm (iii)}$ holds, then $\Pi$ is strongly positively independent.
\item If $\Pi$ and $\ck \Pi$ are strongly positively independent, $\text {\rm (i)--(iii)}$ are equivalent.  In that case, for $w\in W$ and $\alpha\in \Pi$, we have $w(\alpha)\in \Phi_+$ iff $l(ws_\alpha)\geq l(w)$ where $l$ is the length function of $(W,S)$
\end{num}\end{lemma} 
\begin{proof} For (a), if $x=\sum_{\beta\in \Pi} c_\beta \beta=0$ with $c_\alpha<0$ and $c_\beta\geq 0$ for $\beta\neq \alpha$, then $0=\mpair{x,\ck\alpha}<0$, a contradiction. Part (b) can be  checked
by direct calculation in the dihedral case $\vert \Pi\vert=2$ (the calculations in general easily reduce to those in \cite{Dy2}, see \ref{ss:4.8}). The proof of (b) in general  is by reduction  to the dihedral case by a standard argument
\cite[Proposition 2.1]{Deo} (cf.  the proof of Theorem \ref{ss:abchar}(b)).\end{proof}
\begin{remark} See \cite{Jon} for other results in which  the  set $P$ (or $P\setminus \set{0}$)
naturally appear.\end{remark}

\begin{proposition} \label{ss:4.4} Assume $\Pi$, $\ck\Pi$ are positively independent and that the conditions 
$\text{\rm  \ref{ss:4.3}(i)--(iii)}$  hold.  
  For $\alpha,\beta\in\Pi$, define
$m_{\alpha,\beta}=1$ if
$\alpha=\beta$,
$m_{\alpha,\beta}=\infty$ if $c_{\alpha,\beta}\geq 4$ and
$m_{\alpha,\beta}=m$ if $c_{\alpha,\beta}=4\cos^2\frac{\pi}{m}$ with $m\in \Nat_{\geq 2}$.
Then 
\begin{num}\item $(W,S  )$ and $(W',S  ')$ are isomorphic Coxeter systems with Coxeter matrix $(m_{\alpha,\beta})_{\alpha,\beta\in \Pi}$,
an isomorphism being given by
$\theta:s_{\alpha,\ck\alpha}\mapsto s_{\ck\alpha,\alpha}$ for $\alpha\in \Pi$. 
\item Regarding $\theta$ as an identification, we have
$\mpair{w\alpha,\beta}=\mpair{\alpha,w^{-1}\beta}$ (i.e. the representation of $W$ on $V$ and $V'$ are ``contragredient''). 
\item The bijection $\iota\colon \Pi\mapsto\ck\Pi$ extends to  a $W$-equivariant bijection $\hat\iota\colon \Phi\mapsto \ck \Phi$, 
which we  still denote as  $\alpha\mapsto \ck\alpha$, and which restricts to a bijection
 $\Phi_+\rightarrow \ck\Phi_+$. Further, $w(\alpha)=c\beta$
with $w\in W$,  $\alpha,\beta\in \Phi$, $c\in \real$ implies $w(\ck \alpha)=c^{-1}\ck\beta$.
\end{num}\end{proposition}
\begin{proof}  In (c) in the most common situations,
one  necessarily has  $c=1$ and the result is either trivial (as in \cite{How}) or  the proof uses extra structure
not present here,
as in \cite{Kac}. 
In general, one can proceed as follows. It suffices to show that $w(\alpha)=c\beta$
with $w\in W$,  $\alpha,\beta\in \Pi$, $c\in \real_{>0}$ implies $w(\ck \alpha)=c^{-1}\ck\beta$.
This can be  directly checked in the dihedral case (cf \ref{ss:4.8} again), and then  the general case
reduces to the result in    the dihedral case  by an argument similar to the proof of  \cite[Proposition 2.1]{Deo}.
In fact, if $w\neq 1$, write $w=ws_\gamma$ where $\gamma\in \Pi$ and $l(ws_\gamma)<l(w)$. 
Then write $w=w'w''$ where $w''\in \mpair{s_\alpha,s_\gamma}$ and $l(w)=l(w')+l(w'')$, and $l(w')$ is minimal amongst all such expressions.
We have $l(w''s_\alpha)>l(w'')$, so $w''(\alpha)=p\alpha+q\gamma$ for some $p,q\in \real_{\geq 0}$.
Also, $l(w's_\alpha)\geq l(w)$ and $l(w's_\gamma)\geq l(w')$, so $w'(\alpha), w'(\gamma)\in \Phi_+$.
Now \[\Pi\ni \beta=c^{-1}w'w''(\alpha)=c^{-1}pw'(\alpha)+c^{-1}qw'(\beta),\] which  implies that at most one of $p$, $q$ is non-zero. One easily completes the argument using the result for the dihedral case and induction on $l(w)$. Once one has (c) available, one may abbreviate $s_{\alpha,\ck \alpha}$ as $s_\alpha$ and note $s_{w(\alpha)}=ws_\alpha w^{-1}$, and 
the remainder of the  proofs of (a) and (b) are then standard.\end{proof}

\subsection{}\label{ss:4.5} Maintain the assumptions of \ref{ss:4.2}-\ref{ss:4.4}. The reflection $s_{\alpha,\ck\alpha}\in \text{\rm GL}(V)$ in a root $\alpha\in \Phi$ or  the reflection $s_{\ck\alpha,\alpha}\in \text{\rm GL}(V')$ in the corresponding
coroot $\ck \alpha$ will be denoted just
$s_\alpha$ and regarded as an element of $W$.
 Then $s_{w(\alpha)}=ws_\alpha w^{-1}$,
 so the set of reflections of $(W,S)$ is  $T:=\mset{s_\alpha\mid \alpha\in \Phi^+ }=\mset{ws w^{-1}\mid w\in W, r\in S  }$.
 Using \ref{ss:4.3}(c), one sees that for $\alpha,\beta\in \Phi$, one has $s_\alpha=s_\beta$ iff $\alpha=c\beta$ for some $c\in \real_{\neq 0}$.

We  say that the tuple
$E=(\mpair{ ?,?}\colon V\times
V'\rightarrow
\real, \hat \iota\colon \Phi\rightarrow \ck\Phi)$ is a  root datum
 and that  
 $B=(\mpair{ ?,?}\colon V\times
V'\rightarrow
\real, 
\iota\colon \Pi\rightarrow \ck\Pi)$ is a based   root datum, with $E$ as underlying root datum. 
Note that $E$ is determined by $B$,  while $B$ is determined by $E$ and the subset $\Pi$ of $\Phi$, since
$\iota$ is given by restriction of $\hat\iota$. We call a subset $\Pi$ of $\Phi$  a root basis 
of $E$ if it arises from some based root datum $B$ with underlying root datum $E$ in this way.  As usual, we  call $\Pi\subseteq \Phi_+\subseteq \Phi$
the simple roots, positive roots and roots  respectively, and  
$\ck\Pi\subseteq \ck\Phi_+\subseteq \ck\Phi$  the simple coroots, positive
coroots and coroots, respectively. The matrix 
$\mpair{\alpha,\ck \beta}_{\alpha, \beta\in \Pi}$ will be called the 
possibly non-integral generalized Cartan matrix (NGCM) of $B$. The NGCM's with entries in $\Int$ and finite index sets are precisely the generalized Cartan matrices (GCMs) of \cite{Kac}).

\subsection{} \label{ss:4.6}We say that a based root datum as above is a standard based root datum
if  $V= V'$, $\mpair{?,?}$ is a symmetric bilinear form on $V$,
$\Pi=\ck  \Pi$, and  $\iota=\text{\rm Id}_\Pi$. 
Every Coxeter system is isomorphic to a Coxeter system $(W,S)$ associated to
a standard root datum with $\Pi$ and $\ck \Pi$ linearly independent. The class of standard based root datums
affords the same  class of root systems and reflection representation of Coxeter groups as   considered in \cite{How}. 

\subsection{}\label{ss:4.7} We say that a root system (or a corresponding root datum or based root datum) datum  is reduced if 
for for any $\alpha\in \Phi$ and $c\in \real$ with $c\alpha\in \Phi$, one has $c=\pm 1$
i.e. if $\Phi_+$ is the set of positive roots associated to some chosen root basis $\Pi$, then the map $\alpha\mapsto s_\alpha\colon \Phi\rightarrow T$ is a bijection. The root datum is reduced iff for
any root basis $\Pi$, one has $\mpair{\alpha,\ck\beta}=\mpair{\beta,\ck\alpha}$ for all $\alpha,\beta\in \Pi$ with $m_{\alpha,\beta}$ finite and odd (by \ref{ss:3.14a}, the proof of this reduces  to the case of  dihedral Coxeter systems, where it follows by  a simple computation using the remarks at the end of \ref{ss:4.8}). In particular, a root datum 
for which the NGCM is symmetric or is a GCM is reduced in this sense.  

\subsection{}\label{ss:4.8} Let $B$ be a based root datum with simple roots $\Pi$, and $c_\alpha$ for $\alpha\in \Pi$
be non-negative scalars. Then there is a new based datum
 $B'=(\mpair{ ?,?}\colon V\times
V'\rightarrow
\real, \hat \iota'\colon \Phi'\rightarrow \ck{\Phi'},
\iota'\colon \Pi'\rightarrow \ck{\Pi'})$ where $\Pi'=\mset{c_\alpha\alpha\mid \alpha\in \Pi}$ and
$\iota'(c_\alpha \alpha)=c_\alpha^{-1}\iota(\alpha)$. We say $B'$ is obtained by rescaling
$B$ (or by rescaling $\Pi$). If $B'$ can be chosen so as to have a symmetric NGCM, we say $B$ is symmetrizable.

It is easy to see that there are root datums which can not be rescaled to any reduced root datum
(so in particular,  they are not symmetrizable). However, if the Coxeter graph of $(W,S)$ is a tree, then any root datum is symmetrizable. In particular, this applies if $(W,S)$ is dihedral. 
Since the root systems of dihedral groups with symmetric NGCM are described in \cite{Dy2}, one  obtains
easily a description of  root systems of arbitrary based root datums affording dihedral Coxeter systems.

 The following fact  will play an
important role in the proof of the main result of this paper. 
 \begin{theorem}\label{ss:4.9}
\begin{num}
\item  Let
$W'$ be any reflection subgroup of
$W$. There is a  based root datum
$B'=(\mpair{?,?}\colon V\times
V'\rightarrow
\real,\iota'\colon \Pi'\rightarrow \ck{\Pi'})$
 with associated Coxeter system $(W',S')$  such that $\Pi'\subseteq \Phi_+$ and $\iota'$ is the
restriction of $\hat\iota\colon \Phi\rightarrow \ck{\Phi}$. 
\item For any $B'$ satisfying $\text{\rm (a)}$,
 $S'=\chi(W')$. Hence $B'$ is unique up to rescaling, and it is unique if $B$ is a reduced root datum.
\item A subset $\Pi'$ of $\Phi_+$ arises  as in \text{\rm (a)} from some  reflection subgroup $W'$ iff
the conditions of $\text{\rm  \ref{ss:4.3}(iii)}$  hold for all $\alpha\neq \beta\in \Pi'$.
Further, if these conditions hold, $W'=\mpair{s_\alpha\mid \alpha\in \Pi'}$.
\end{num}\end{theorem}
\begin{proof} In the case of standard based root datums, this is proved 
in \cite{Dy2} and \cite{Dy0}. Either  proof extends 
mutatis mutandis to the more general  situation here, using Lemma \ref{ss:4.3}. \end{proof}

In the above setting, we call $B'$ a based root subdatum of $B$ corresponding to $W'$.
\section{Comparison of real and abstract root systems}

\begin{proposition}\label{ss:5.1}  Let $B=(\mpair{ ?,?}\colon V\times
V'\rightarrow
\real, 
\iota\colon \Pi\rightarrow \ck\Pi)$ be a based   root datum with associated Coxeter system
$(W,S)$, and $(\widehat T,F)$ be the standard abstract root system of $(W,S)$. 
Let  $F'\colon \Phi\rightarrow \text{\rm Sym}(\Phi)$ be defined by $F'(\alpha)= (s_\alpha)_{\vert \Phi}$
(the restriction of $s_\alpha$ to $\Phi$).
\begin{num}\item 
  $(\Phi,F')$ is an abstract root system with $\Phi_+$ as generative quasi-positive system and  with associated Coxeter system naturally isomorphic to $(W,S)$. 
\item There is   a surjective $W$-equivariant map $\theta\colon \Phi\rightarrow \widehat T$ determined by
$\alpha\mapsto (s_\alpha,\epsilon)$ for $\alpha\in \epsilon \Phi_+$ and $\epsilon\in \set{\pm 1}$. Further, $\theta$ determines  an isomorphism
$(\Phi,F')\cong (\widehat T,F)$ iff $\Phi$ is reduced. 
\item Let $\alpha,\beta,\gamma\in \Phi$. Then 
$\gamma\in \real_{\geq 0}\alpha+\real_{\geq 0}\beta$ iff $\theta(\gamma)$ is between  $\theta(\alpha)$ and $\theta(\beta) $. \end{num} \end{proposition}
\begin{proof} Parts (a)--(b) are clear from the definitions and results in Section 1--2. For part (c), the definitions and \ref{ss:4.9} immediately reduce the proof in general
to that in the case of dihedral Coxeter systems.
In the dihedral case, there is an obvious bijection from the set of rays spanned by the roots
to the set of rays in the diagram \ref{ss:3.22}, mapping $\real_{\geq 0}\alpha$ to the ray labelled $\theta(\alpha)$ for $\alpha\in \Phi$. By direct calculations for dihedral root systems (which are simplified by using the remarks at the end of  \ref{ss:4.8}), one checks that this  bijection preserves the sets of rays in the convex closure  of  a union of  two rays. 
\end{proof}
 
 \begin{theorem}\label{ss:5.2} Let notation be as in the preceding proposition. Let $(W_i,S_i)$ for $i\in I$ be the irreducible components of $(W,S)$, $\Pi_i=\mset{\alpha\in \Pi\mid s_\alpha\in S_i}$ and $V_i:=\real\Pi_i$
 denote the $\real$-vector space spanned by $\Pi_i$.
 Assume that the sum $\sum_i\real V_i$ of subspaces of $V$ is direct
 (e.g. $(W,S)$ is irreducible or $\Pi$ is linearly independent) and that the dual condition (with $V$ replaced by $V'$) also holds.
Then a subset $\Pi'$ of $\Phi$ is a root basis of $E$, affording a based root datum $B'$, say,   iff the restriction $\theta_{\vert \Pi'}$   is injective and  $\theta(\Pi')$  is an abstract root basis of $(\widehat{T},F)$.
In that case, the set of positive roots of $B'$ is $\theta^{-1}(\Psi_+)$ where $\Psi_+$ is the positive system of $(\widehat{T},F)$corresponding to $\theta(\Pi')$. \end{theorem}
   \begin{proof} If $\Pi'$ is a root basis, then $\theta_{\vert \Pi'}$ is injective, $\theta(\Pi')$ is an  abstract root basis of $(\widehat T,F)$, and the positive roots of $B'$ are $\theta^{-1}(\Psi_+)$, using \ref{ss:5.1}.
   
    Suppose now that  $\theta_{\vert \Pi'}$ is injective and $\theta(\Pi')$ is an abstract root basis of $(\widehat T,F)$.     
     For any $\alpha,\beta\in \Pi'$, let $W_{\alpha,\beta}:=\mpair{s_\alpha,s_\beta}$
and $\widehat T_{\alpha,\beta}=(W_{\alpha,\beta}\cap T)\times\set{\pm 1}\subseteq \widehat {T}$. Using
\ref{ss:3.28}(c), $\theta(\set{\alpha,\beta})$ is an abstract set of simple roots for $W_{\alpha,\beta}$ in the abstract root system $T_{\alpha,\beta}$ of $W_{\alpha,\beta}$. This implies (by direct calculation for the dihedral groups, cf. \ref{ss:3.25}) that $\set{\theta(\alpha),\theta(\beta)}$ is conjugate by an element $(w,\epsilon)$ of  $W_{\alpha,\beta}\times \set{\pm{1}}$ to $\set{s_\gamma,s_\delta}\times{1}$, where $\gamma,\delta\in \Phi_+$ with $\set{s_\gamma,s_\delta}= \chi(W_{\alpha,\beta})$.  We have, say
$w(\alpha)=c\gamma$, $w(\beta)=d\delta$ for some  $c,d\in \real$ with $cd>0$ (since $c$, $d$ can be taken of the same sign as $\epsilon$).
Now $\mpair{\gamma,\ck\delta}$ and $\mpair{\delta,\ck \gamma}$ are non-positive real numbers
whose product is in the set $P$ of \ref{ss:4.3}, and such that both are zero if either is zero.
Since $w(\ck \alpha)=c^{-1}\ck \gamma$ and $w(\ck \beta)=d^{-1}\ck\delta$, it follows that
$\mpair{\alpha,\ck\beta}$ and $\mpair{\beta,\ck \alpha}$ have the same properties. 
Since $W=\mpair{s_\alpha\mid\alpha\in \Pi'}$, we conclude from the definitions in  \ref{ss:4.5} that $B'$ is a based  root datum 
provided that $\Pi'$ is positively independent ($\ck {\Pi'}$ will then be positively independent by symmetry).

We prove positive independence of $\Pi'$ first only under the additional hypothesis
that $\Pi$ is finite and linearly independent; this extra hypothesis will be replaced by the more general one of the theorem in \ref{ss:5.7}.
So assume $\Pi$ is finite and linearly independent. Let $S'=\mset{s_\alpha\mid \alpha\in \Pi'}$.
Then $(W,S)$ and $(W,S')$ are   Coxeter systems with $S'\subseteq T$, so $\vert S\vert=\vert S'\vert$
from \ref{ss:3.7}
and thus $\vert \Pi\vert=\vert \Pi'\vert$. Since $\mpair{S'}=W$ with $S'\subseteq T$, every reflection of $W$ (in particular, any element of $S$) is equal to  a reflection in some element of  $\mpair{S'}\Pi'$. It follows that $\Pi\subseteq \real \Pi'$  and hence $\Pi'$ is a $\real$-basis of $\real\Pi$. Since $\Pi'$ is  linearly independent, it  is positively independent as required.
\end{proof}

It will be  helpful to keep in mind in the following  that if $(W,S)$ and $(W,S')$ are two Coxeter systems
with $S'\subseteq T$, then $(W,S)$ is of finite rank (resp., is irreducible) iff $(W,S')$ is of finite rank (resp., is irreducible).

\begin{theorem}\label{ss:5.3} Suppose that $B=(\mpair{ ?,?}\colon V\times
V'\rightarrow
\real, 
\iota\colon \Pi\rightarrow \ck\Pi)$ and 
 $B'=(\mpair{ ?,?}\colon V\times
V'\rightarrow
\real, 
\iota'\colon \Delta\rightarrow \ck\Delta)$ are two based root datums
with the same underlying root datum
 $(\mpair{ ?,?}\colon V\times
V'\rightarrow
\real, \hat \iota\colon \Phi\rightarrow \ck\Phi)$. Assume that  
the associated  Coxeter systems $(W,S  )$ and $(W,S  ')$ respectively 
are both of finite rank and irreducible.Then
there exists $w\in W$, $\epsilon\in\set{\pm 1}$ and
scalars $c_\alpha\in \real_{>0}$ for $\alpha\in \Pi$ such that
$\Delta=\mset{\epsilon c_\alpha w(\alpha)\mid \alpha\in \Pi}$.
In particular, $S'=wSw^{-1}  $, and hence $(W,S  )$ is isomorphic to $(W,S  ')$. \end{theorem}

\begin{proof}  Suppose that $B$ is a standard based root datum. It is clear  that $B'$ must be a standard based root datum also. In this case, the conclusion is one of the main results of  \cite{How}
(with all $c_\alpha=1$ necessarily in this case). 
If $B$ is not a standard based root datum, the result will be proved in \ref{ss:5.6}.  \end{proof}

\begin{remark} Slightly more generally, one could assume in Theorem \ref{ss:5.3} that the underlying root datums of $B$ and $B'$ are not necessarily equal, but differ by rescaling. This version reduces immediately to the one above by rescaling $B'$, say, appropriately.
\end{remark} 
 \begin{theorem} \label{ss:5.4} Suppose that $(W,S)$ is an irreducible Coxeter system of finite rank.  If ${\Psi_+}$ is an abstract system of positive roots for $\widehat T$, then there exists $w\in W$ and $\epsilon\in \set{\pm 1}$ with ${\Psi_+}=\epsilon w( \widehat T_+)$. In particular, $\Delta_{\Psi_+}=\epsilon w(S_+)$,
  $S_{\Psi_+}=wSw^{-1}$ and $(W,S_{\Psi_+})$ is isomorphic to $(W,S)$ as Coxeter system.\end{theorem}

\begin{proof} We may suppose without loss of generality that
$(W,S)$ is the Coxeter system associated to a standard based root datum $B$ such that $\Pi$ is linearly independent. Note that both \ref{ss:5.2} and \ref{ss:5.3} have already been proved for $B$ of this special type.
Set $\Pi'=\theta^{-1}(\Delta_{\Psi_+})$. Now $\theta$ is a bijection. By \ref{ss:5.2},
$\Pi'$ is a root basis of the root datum $E$ underlying $B$. By \ref{ss:5.3},
$\Pi'=\epsilon w(\Pi)$ for some $w\in W$ and  $\epsilon\in \set{\pm 1}$. 
Then $\Delta_{\Psi_+}=\theta(\Pi')=\epsilon w\theta(\Pi)=\epsilon w(S\times\set{1})$ as required. 
 \end{proof} 
   \begin{remark} \label{ss:5.5}It is well-known that $\widehat T_+$ is conjugate to $-\widehat T_+$ iff $W$ is finite. 
Further, $w(\widehat {T}_+)=\widehat T_+$ iff $w=1$. It follows that under the assumptions of   Theorem \ref{ss:5.4}, if  $W$ is finite (resp., infinite) then $W$ (resp., $ W\times\set{\pm 1}$) acts simply transitively
  on both the set of abstract positive root systems for $\widehat T$ and the set of
  abstract root bases of $\widehat T$. 
Hence in \ref{ss:5.3} or \ref{ss:5.4},  $w$, $\epsilon$ are uniquely determined
provided $\epsilon=1$ if $W$ is finite. \end{remark} 
 
 \subsection{Completion of proof of \ref{ss:5.3} in general} \label{ss:5.6} Let $\theta$ be as in \ref{ss:5.1}.
 By  (an already proven) part of \ref{ss:5.2}, $\theta(\Pi')$ and $\theta(\Pi)$ are two abstract root bases of
 $\widehat{T}$. By \ref{ss:5.4}, there is $w\in W$ and $\epsilon \in \set{\pm 1}$ such that 
 $\theta(\Pi')=\epsilon \theta(\Pi)$. This implies that  $\Pi'$ and $\epsilon w(\Pi)$ are the same up to rescaling, which is what we needed.
 \subsection{Completion of proof of \ref{ss:5.2} in general} \label{ss:5.7}
 Assume that  $\sum_i \real \Pi_i$ is a direct sum and   that the dual condition also holds.  Let  $\Pi'$ is a subset of $\Phi$ such that $\theta_{\vert \Pi'}$ is injective and  $\theta(\Pi')$ is an abstract root basis of $\widehat T$.
 To complete the proof, it remains to show that $\Pi'$ is positively independent. 
  Using \ref{ss:reduc} and the hypothesis that $\sum_i \real \Pi_i$ is direct, one readily reduces to the case that $(W,S)$ is irreducible, as we assume henceforward.
 
 Let $R$ be any finite subset of $S'$ such that  the Coxeter system $(\mpair{R},R)$ is irreducible. It will suffice to show that  
$\Delta=\mset{\alpha\in \Pi'\mid s_\alpha\in R}$ is positively independent, for any such $R$
(since $\Pi'$ is positively independent if all its finite subsets are positively independent).
Now let $D$ be a based root subdatum of $B$ corresponding to $W':=\mpair{R}$,   as in \ref{ss:4.9}. By rescaling $D$ if necessary, we may assume further that some root basis $\Delta'$ of $D$ is contained in $W'\Delta$. By a further rescaling of $\Delta$ if necessary, we may assume that
 $\Psi:=W'\Delta'=W'\Delta''$ is the root system of $D$. Let $\Psi_+=\Psi\cap \Phi_+$ denote the positive roots of $D$. Let $\widehat{ T'}$ be the abstract root system of $W'$
 and $\theta'\colon \Psi\rightarrow \widehat{T'}$ be the analogue for $D$ of the map $\theta\colon \Phi\rightarrow \widehat{T}$ for $B$ 
(clearly, $\theta'$ is given by restriction of $\theta)$. Since $\theta(\Pi')$ is an abstract root basis of 
$\widehat{T}$, it follows easily that $\theta'(\Delta)$ is an abstract root basis of $\widehat{T'}$.
 But, by \ref{ss:4.9} and \ref{ss:5.1}, $\theta'(\Delta')$ is also an  abstract root basis of $\widehat{T'}$. Since $(W',R)$ is irreducible and of finite rank by assumption, and $R\subseteq T$,  $(W',\chi(W'))$ is irreducible of finite rank also. Hence $\theta'(\Delta)=\epsilon w\theta'(\Delta')$ for some $\epsilon\in \set{\pm 1}$ and $w\in W'$, by \ref{ss:5.4}. Rescaling $\Delta'$ again if necessary,  we may assume that $\Delta=\epsilon w(\Delta')$. Since $\Delta'$ is positively independent, so is $\Delta$.

    \begin{corollary}\label{ss:abchar2} Suppose that $(W,S)$ is of finite rank. Let $\Delta\subseteq \widehat T$. 
  Then $\Delta$ is an abstract root basis of $\widehat T$ iff the conditions
   $\text{\rm \ref{ss:abchar}(i)--(ii)}$ hold and $\vert \Delta\vert=\vert S\vert$.
  \end{corollary}
  \begin{proof} If $\Delta$ is an abstract root basis, then  \ref{ss:abchar}(i)--(ii) hold and $\vert \Delta\vert=\vert S\vert$. For the converse, assume that  \ref{ss:abchar}(i)--(ii) hold and $\vert \Delta\vert\leq \vert S\vert$. 
  We may assume that $(W,S)$ is realized as the Coxeter system associated to a standard based root datum $B$ with linearly independent root basis $\Pi$ and root system $\Phi$. Let $\theta$ be as in \ref{ss:5.1}(b), and  $\Pi'=\theta^{-1}(\Delta)\subseteq \Phi$.
  To show that $\Delta$ is an abstract root basis of $\widehat T$, it is enough by Theorem \ref{ss:5.2} to show that $\Pi'$ is a root basis of $B$. Using \ref{ss:4.9}, we see from the assumptions \ref{ss:5.2}(i)-(ii), that it would be sufficient to show that  $\Pi'$ is positively independent.  
      Since $\mpair{S'}=W$ where
  $S'=\mset{s_\alpha\mid \alpha\in \Pi'}$, we have $\mpair{S'}\Pi'=\Phi$ and  a similar  argument to that in the last paragraph of the proof of  \ref{ss:5.2} shows that $\Pi'$ is a basis of $\real \Pi$ and hence $\Pi'$ is positively independent. 
 \end{proof}

 \begin{example} \label{ss:5.8} We show that Theorem \ref{ss:5.2} may fail without the hypothesis that
 $\sum_iV_i$ is direct. 
 Consider a real vector space $V$ with basis $\set{e_i}_{i\in I} \cup \set{\delta}$, equipped with a symmetric bilinear form $\mpair{?,?}\colon V\times V\rightarrow \real$ determined by 
 $\mpair{\delta,V}=0$ and $\mpair{e_i,e_j}=2\delta_{i,j}$.
 Set $\Pi=\mset{e_i,\delta-e_i\mid i\in I}$, $\Pi'=\Pi$, $\iota=\Id_{\Pi}\colon \Pi\rightarrow \Pi'$.
 This data determines  a standard based root datum $B$ as in \ref{ss:4.6}. Let $(W,S)$ be the corresponding  Coxeter system, so $S=\cup_{i\in I} S_i$ (disjoint union) where $S_i:=\mset{s_\alpha\mid \alpha\in \Pi_i}$, $\Pi_i:=\set{{e_i}, {\delta-e_i}}$. In fact, the irreducible components of $(W,S)$ are $(W_i,S_i)$ (of type $\tilde A_1$) for $i\in I$, where $W_i:=\pair{S_i}$. Now for any signs $\epsilon_i\in \set{\pm 1}$, 
 $\Gamma:=\cup_{i\in I} S_i\times\set{\epsilon _i}$ is clearly an abstract root basis of $\widehat T$
 by Example \ref{ss:3.14}.
 However, $\theta\colon \Phi\rightarrow \widehat T$ is bijective, and
 $\Pi'=\theta^{-1} (\Gamma)=\cup_{i\in I} \epsilon \Pi_i$ is not a root basis for $B$ if $i\mapsto \epsilon_i$ is not a constant function. For if $i,j\in I$ with $\epsilon_ i=-\epsilon_ j$, then
 $\epsilon_ i e_i +\epsilon _i(\delta-e_i)+\epsilon _j e_j+\epsilon_j(\delta-e_j)=0$  and  $\Pi'$ is not positively independent. 
 
We remark that if $\vert I\vert =4$, the above based root datum $B$ is isomorphic to a based root subdatum   of the standard (as in \cite{Bou} or \cite{Hum}) based root datum of the  affine Weyl group of type $\tilde D_4$
(cf  \cite{Dy2}).\end{example}

\begin{example} \label{ss:5.10} Prior to  studying  conjugacy of root bases of infinite rank Coxeter systems,
we discuss in more detail the two  Coxeter systems 
which play an exceptional role  in Theorem 1. Note  that these two
 Coxeter systems, of type $A_{\infty,\infty}$ and $A_\infty$,  have Coxeter graphs
\[\xymatrix{\ldots\ar@{-}[r]&\bullet\ar@{-}[r]&\bullet\ar@{-}[r]&\bullet\ar@{-}[r]&\ldots&&\bullet\ar@{-}[r]&\bullet\ar@{-}[r]&\bullet\ar@{-}[r]&\ldots}\] 
 respectively, so are obviously non-isomorphic. 
 
 Fix a Coxeter system $(W,S)$ of type $A_{\infty,\infty}$.
We may identify  
$W$  with the group of all permutations of $\Int$ which fix all but finitely many integers, so that 
 $S=\mset{s_n=(n,n+1)\mid n\in \Int}$ (the set of adjacent transpositions).
Let  $K=\mset{s_n\mid n\in \Nat}\subseteq S$. Then the standard  parabolic subsystem $(W_K,K)$ is a Coxeter system of type $A_\infty$.
Clearly, by restriction of its action to $\Nat\subseteq \Int$, 
  $W_K$ identifies  with the group of all permutations of $\Nat$ which fix all but finitely many integers, with  Coxeter generators $\mset{s_n=(n,n+1)\mid n\in \Nat}$.  Then the reflections of $(W,S)$ (resp., $(W_K,K)$) are the transpositions  $T=\mset{(i,j)\mid i<j\text{ \rm in $\Int$}}$ (resp., $T'=\mset{(i,j)\mid i<j\text{ \rm in $\Nat$}}$). 
  
   Clearly, any bijection $\sigma :\Nat\xrightarrow{\cong} \Int$ induces a group isomorphism $\sigma'\colon W_K\rightarrow W$ defined by $ w\mapsto \sigma w\sigma^{-1}$
which restricts to a bijection $T\rightarrow T'$. Similarly,  any permutation $\tau$ of $\Int$ (resp., $\Nat$)  induces an 
(in general outer) automorphism $\tau'$ of $W$ (resp., $W_K$) defined by $ w\mapsto \tau w\tau^{-1}$ and 
 which restricts to a permutation of $T$ (resp., $T')$.

    Let $V$ be a real vector space on basis $e_n$ for $n\in \Int$,
   with (positive definite) symmetric bilinear form defined by
   $\mpair{\sum_n a_ne_n, \sum_n b_ne_n}=\sum_n a_nb_n$.
   Let $\Pi=\mset{e_n-e_{n+1}\mid n\in \Int}\subseteq V$. Then $(\mpair{?,?}\colon V\times V\rightarrow \real,\Id_\Pi\colon \Pi\rightarrow \Pi)$
   is a  standard based root datum affording a Coxeter system $(W,S)$ of type $A_{\infty,\infty}$. 
 (The set $X=\mset{e_n\mid n\in \Int}\subseteq V$ is stable under the $W$-action. Restriction of the $W$-action to $X$ and identifying $X$ with $\Int$ 
   via the bijection $n\mapsto e_n\colon \Int\rightarrow X$,  provides the  description of $W$ as a subgroup of $\text{\rm Sym}(\Int)$ from above. Similarly,  $Y=\mset{e_n\mid n\in \Nat}\subseteq V$
   is stable under the action of $W_K$ and affords in a similar way the embedding $W_K\subseteq \text{\rm Sym}(\Nat)$).    

   For any $J\subseteq S$, let $\Pi_J:=\mset{\alpha\in \Pi\mid s_\alpha\in J}$, $V_J:=\real \Pi_J$ denote the $\real$-span of $\Pi_J$, $W_J:=\mpair{J}$, $\Phi_J:=W_J\Pi_J\subseteq V_J$ and $\mpair{?,?}_J$ denote the restriction of $\mpair{?,?}$ to a bilinear form on $V_J$. Then
   $B_J:=(\mpair{?,?}_J, \Id_{\Pi_J}\colon \Pi_J\rightarrow \Pi_J, 
   )$
   is a standard based root datum for the standard parabolic subsystem $(W_J,J)$ of $(W,S)$.
   Let $E_J$  denote the root datum underlying the based root datum $B_J$, with root system denoted as $\Phi_J$.

 Now fix a  bijection  $\sigma\colon \Nat\rightarrow \Int $. This induces an isomorphism
 $\tilde \sigma\colon \real Y\rightarrow \real X$ of vector spaces  mapping basis elements
 by $e_n\mapsto e_{\sigma(n)}$. It is straightforward to check that
 $\tilde\sigma(\Pi_K)$ (resp., $\tilde\sigma^{-1}(\Pi_S)$) is a root basis of $E_S$ (resp., $E_K$)
 with respect to which the associated Coxeter system is of type $A_\infty$ (resp., $A_{\infty,\infty}$).
 Similarly, a permutation $\tau$ of $\Int$ (resp., $\Nat$) induces an $\real$-linear automorphism $\tilde \tau$  of 
 $\real X$ (resp., $\real Y$)  determined by  $e_n\mapsto e_{\tau(n)}$, and $\tilde \tau(\Pi_S)$ (resp., $\tilde \tau(\Pi_K)$ is a root basis of $E_S$ (resp., $E_K$) which is not in general $W$-conjugate to $\Pi_S$ (resp., $\Pi_K$) up to sign, though the Coxeter system associated to this root basis is still of type $A_{\infty,\infty}$ (resp., $A_\infty$).
 
  \end{example}
\begin{remark} The above   example comes from a similar isomorphism of infinite-rank Kac-Moody Lie algebras
 with Dynkin diagrams as above. Namely, such an algebra of type $A_{\infty,\infty}$ (resp., $A_{\infty}$) 
 may be realized as the complex  Lie algebra of all $\Int\times\Int$-indexed  (resp., $\Nat\times\Nat$-indexed) complex matrices with only finitely many non-zero entries, and trace zero. A bijection $\Int\rightarrow \Nat$ induces an isomorphism between these  Lie algebras in the obvious way, inducing the above isomorphism of root systems in the (restricted dual) of the corresponding Cartan subalgebras.\end{remark}
 \begin{definition} We say that a map  $\lambda\colon \widehat T\rightarrow \widehat T$   is given locally by the action of $W$ if for each finitely generated reflection subgroup $W'$ of $W$,
 there is an element $w=w(W')$ of $W$ such that $\lambda(\alpha)=w(\alpha)$ for all
   $\alpha\in (W'\cap T)\times\set{\pm 1}$.  Such a map $\lambda$ is injective and preserves betweenness; if $\lambda$ is  invertible, its inverse is also given locally by the action of $W$.  We let $\widehat  W$ denote the group of all
   permutations of $\widehat T$ which are given locally by the action of $W$, under composition.
   Then $W\subseteq \widehat W$, and equality holds if  $(W,S)$ is of finite rank. 

  \end{definition}

   \begin{theorem}\label{ss:5.13} Let $(W,S)$ be an irreducible Coxeter system with standard abstract root system $\widehat T$, and  let $\Psi_+$ be a  positive system of $\widehat{T}$. \begin{num}\item $(W,S_{\Psi_+})$ is isomorphic to $(W,S)$ unless one of them is of type $A_{\infty}$ and the other is of type $A_{\infty,\infty}$
   \item If $(W,S_{\Psi_+})\cong (W,S)$, there is $\hat w\in \widehat W$ and $\epsilon\in \set{\pm 1}$
   with $\Psi_+=\epsilon \hat w(\widehat{T}_+)$
   \item Any betweenness-preserving automorphism $\sigma$ of  the abstract root system of $(W,S)$
   is expressible (not necessarily uniquely) as $\sigma =\epsilon' \hat w d'$ where $\hat w\in \widehat W$, and 
  $d'$ (resp., $\epsilon '$) is an automorphism induced by a diagram  automorphism of $(W,S)$ (resp., by the action of $\epsilon\in \set{\pm 1}$). \end{num}
 \end{theorem}

  \subsection{}    We need to give  some  generalities on  an  (possibly infinite rank) 
  irreducible Coxeter system $(W,S)$ prior to  the proof of \ref{ss:5.13}. Fix such a Coxeter system with reflections $T$ and standard  abstract root system $\widehat T$. We distinguish two cases throughout the following discussion.
  We say that $(W,S)$ is of type  (LocFin) if  every finite rank, irreducible  standard parabolic subgroup of $(W,S)$ is finite. We say $(W,S)$ is of type  (LocInf) if  $(W,S)$ has some infinite, irreducible, finite rank, standard parabolic subgroup i.e. if it is  not of type (LocFin).
  
  According to a well-known result of Tits (see \cite[Theorem 2.5]{How})   any finite subgroup of $W$ is  contained in some finite  parabolic subgroup of $W$, so the maximal (under inclusion)  finite subgroups of $W$ coincide with the maximal (under inclusion)  finite parabolic subgroups of $W$. Using this, one sees that $(W,S)$  is  of type (LocInf)  iff $(W,S)$ has some  infinite, finite rank (irreducible) reflection subgroup; in particular,  the type of $(W,S)$  depends only on the pair $(W,T)$. 
  
  If $(W,S)$ is of type  (LocFin) (resp., (LocInf)) we let $\mc{S}'$ denote the set of all irreducible, finite rank reflection subgroups of $(W,S)$ (resp., infinite, finite rank irreducible  reflection subgroups of $(W,S)$), ordered by inclusion. Note that $\mc{S}'$ depends only on the pair $(W,T)$.
To allow greater convenience of application of some of the following results,  we work below not just with $\mc{S}'$, but more generally with any cofinal subset  $\mc{S}$ 
  of   $\mc{S}'$.  (Recall that a subset $X$ of a poset $Y$ is
  said to be cofinal in $Y$ if for any $y\in Y$, there is $x\in X$ with $y\leq x$). Since $(W,S)$ is irreducible,
  $\mc{S}$ is cofinal in the (inclusion ordered) family of all finite rank reflection subgroups of $(W,S)$.
  Also $\mc{S}$ is directed; given $W_1,W_2\in \mc{S}$, there is $W_3\in \mc{S}$ with $W_3\supseteq W_1\cup W_2$.

    For any reflection subgroup $W'$ of $W$, let  $\widehat {T'}=(W'\cap T )\times\set{\pm 1}$. For any subset  $\Psi+$ of $\widehat T$, set $\Psi_+(W'):=\widehat{T'} \cap \Psi_+$, so
    $\Psi_+=\cup_{W'\in \mc{S}}\Psi_+(W')$. For example, in this notation, Lemma \ref{ss:3.28}(c)
    asserts that if $\Psi_+$ 
   is an abstract positive system for $\widehat T$, then  $\Psi_+(W')$ is an abstract positive system
   for $\widehat{T'}$. 
   
    Below, by a limit $\lim_{W'\in S} F(W')$  where $F(W')$ is a set for $W'\in \mc{S}$, we mean the set of all $t\in \cup_{W'\in \mc{S}} F(W')$ for which there is $W'\in \mc{S}$ such that for all $W''\in \mc{S}$ with $W''\supseteq W'$, we have $t\in F(W'')$. 

  \begin{lemma} \label{ss:5.16} Let $(W,S)$ be an irreducible  Coxeter system with standard abstract positive system $\Psi_+^0=\widehat{T}_+$. 
 Then a subset $\Psi_+\subseteq \widehat{T}$ is an abstract system
 of positive roots for $\widehat{T}$ iff  for all $W'\in \mc{S}$, $\Psi_+(W')= \epsilon a_{W'}\Psi^0_+(W')$
  for some sign $\epsilon\in \set{\pm 1}$ and function $W'\mapsto a_{W'}\colon \mc{S}\rightarrow W$ satisfying the following conditions 
  \begin{conds}\item  $a_{W'}\in W'$ 
  \item  for $W',W''\in \mc{S}$ with $W'\subseteq W''$, $a_{W'}^{-1}a_{W''}$ is the unique element of minimal length of the coset $W'a_{W''}$  with respect to the standard length function of  $(W,S)$, 
  \item  $S':=\lim_{W'\in \mc{S}}\,a_{W'}\chi(W')a_{W'}^{-1}$ generates $W$.\end{conds} 
Here the $a_{W'}$ and $\epsilon$ are uniquely determined, provided $\epsilon=1$ 
in case $\text{\rm (LocFin)}$.  If the conditions hold, the  abstract root basis  corresponding to $\Psi_+$ is
 $\Delta_{\Psi_+}= \lim_{W'\in \mc{S}}\epsilon a_{W'}(\chi(W')\times \set{1})$, and $S'=S_{\Psi_+}$.
 \end{lemma}

\begin{proof} Consider first two  abstract positive systems $\Psi^i_+$ for $i=0,1$ with $\Psi_+^0=\widehat{T}_+$.
 We show  that there is 
 a unique function $W'\mapsto a_{W'}\colon \mc{S}\rightarrow W$ with $a_{W'}\in W'$
 and a unique  sign $\epsilon \in \set{\pm 1}$ such that $\epsilon=1$ in case (LocFin),  and such that  for all $W'\in \mc{S}$,
$\Psi_+^1(W')=\epsilon a_{W'}(\Psi_+^0(W'))$.
 In case (LocFin), this follows immediately from \ref{ss:5.4}--\ref{ss:5.5}. 
In case (LocInf), we have also from \ref{ss:5.4}--\ref{ss:5.5} that for $W'\in \mc{S}$,  
$\Psi_+^1(W')=\epsilon_{W'} a_{W'}(\Psi_+^0)$ for a unique sign $\epsilon_{W'}$ and
 $a_{W'}\in W'$, and we have to show that $\epsilon_{W'}$ is independent of $W'$. 
 Since $\mc{S}$ is directed, it suffices to show $\epsilon_{W'_1}=\epsilon_{W'_2}$ if $W_1'\subseteq W_2' \in \mc{S}$. But $\epsilon_{W'}=1$ iff $\Psi_+^1(W')\cap -\Psi_+^0(W')$ is finite,
 $\epsilon_{W'}=-1$ iff $\Psi_+^1(W')\cap \Psi_+^0(W')$ is finite
  and $\Psi_+^i(W_1')\subseteq \Psi_+^i(W'_2)$ if $W_1'\subseteq W_2'$, so this is clear.
  
  Now suppose that $W'\subseteq W''$ in $\mc{S}$. Since $\Psi^1_+(W'')\supseteq \Psi^1_+(W')$, we have \begin{equation}\label{eq:3} a_{W''}^{-1}a_{W'}\Psi^0_+(W')\subseteq \Psi^0_+(W'')\subseteq \widehat{T}_+
  \end{equation}
  which gives (ii). 
  
  Next, we show $\Delta_{\Psi^1_+}=\Delta'$ where $\Delta':= \lim_{W'\in \mc{S}}\epsilon a_{W'}(\chi(W')\times \set{1})$. Note $\Delta'\cap -\Delta'=\emptyset$ and that $S'$ defined as in (iii) also satisfies $S'=\mset{s_\alpha\mid \alpha\in \Delta'}$. Now  if $\alpha\in \Delta_{\Psi^1_+}$, then for any $W''\in \mc{S}$ with
  $\mset{1,s_\alpha}\subseteq W''$, $\alpha$ is an abstract simple root  for $\Psi^1_+(W'')$
  (since it is one for $\Psi^1_+$), so $\alpha\in \epsilon a_{W''}(\chi(W'')\times\set{1})$.
  This shows that $\Delta_{\Psi^1_+}\subseteq \Delta'$. Hence $S_{\Psi_+}\subseteq S'$ both generate $W$. Since any relation on $S'$ in $W$ involves elements of only a finite subset of $S'$, it  is easy to see that $S'$ is a set of Coxeter generators for $W$, so $S'=S_{\Psi_+^1}$ by \ref{ss:mingen}.  Since $\Delta'\cap -\Delta'=\emptyset$,
  this gives $\Delta_{\Psi^1_+}=\Delta'$.
  
  Conversely, suppose given $\epsilon\in\set{\pm 1}$  and $a_{W'}\in W$ for $W'\in \mc{S}$ satisfying
  (i)--(iii). Define the sets $X_{W'}:=\epsilon a_{W'}\Psi^0_+(W')$ and
  $\Psi_+:=\cup_{W'\in \mc{S}} X_{W'}$. From (i) and (ii), we deduce that
  $X_{W'}\subseteq (W'\cap T)\times\set{\pm 1}$, $X_{W'}\subseteq X_{W''}$ for $W'\subseteq W''$ in $\mc{S}$ (cf \eqref{eq:3}) and $X_{W'}=\Psi_{+}(W')$. Clearly, $\Psi_+$ is a quasi-positive system.
  Since any dihedral refection subgroup of $(W,S)$ is contained in $W'$ for some $W'\in \mc{S}$,
   it  easily follows  that $\Psi_+$ is biclosed. Now let $\Delta'$ be as in the paragraph immediately  above.
   As before,  $\Delta'\cap -\Delta'=\emptyset$ and $S'=\mset{s_\alpha\mid \alpha\in \Delta'}$.
   If $\alpha\in \Delta'$, then there is some $W'\in \mc{S}$ such that $\alpha$ is a simple root for $\Psi_+(W'')$ for all $W''\supseteq W'$ in $\mc{S}$. Since $\Psi_+=\cup_{\substack{W''\in \mc{S}\\
   W''\supseteq W'} }
   \Psi_+(W'')$, the definitions give $\Delta'\subseteq \Delta_{\Psi_+}$.  On the other hand, if $\alpha\in \Delta_{\Psi_+}$, then $\alpha$ is a simple reflection of $\Psi_+(W'')$ for all $w''$ with $s_\alpha\in W''$,
   so $\alpha\in \Delta_{\Psi_+}$. Hence $\Delta'=\Delta_{\Psi_+}$, $S'=S_{\Psi_+}$ and by (iii), $\Psi_+$ is generative.
    \end{proof}
   
  \begin{corollary} \label{ss:5.17} Let ${\Psi^i_+}$ be abstract systems of positive roots for the arbitrary  Coxeter system
 $(W,S)$, for $i=0,1$. Then $(W,S_{\Psi^1_+})$ and $(W,S_{\Psi_+^0})$ have the same finitely generated parabolic subgroups. 
 \end{corollary}
 \begin{proof} Using  \ref{ss:reduc}, we may reduce to the case that $(W,S)$ is irreducible.
  In that case, it will suffice to show that
 any finite subset $\Delta_1$ of $\Delta_{\Psi_+}^1$ is $W$-conjugate up to sign to  some finite subset
 $\Delta_0$ of $\Delta_{\Psi^0_+}$. By symmetry (see \ref{ss:3.27}), we may assume $\Psi^0_+=\widehat{T}_+$.
 Choose a finitely generated standard parabolic subgroup  $W'=W_K\in \mc{S}'$ of $(W,S)$ such that $\Delta_1\subseteq \Psi_+^1(W')$.
 Then  as in the proof of  \ref{ss:5.16}, $\Delta_1\subseteq \epsilon a_{W'}(K\times \set{1})$. We may take
 $\Delta_0=\epsilon a_{W'}^{-1}\Delta_1\subseteq (K\times \set{1})\subseteq (S\times \set{1})=\Delta_{\Psi_+^0}$. This proves the Corollary.   \end{proof}
  \begin{example}\label{ss:5.18} Suppose $(W,S)$ is of type (LocFin) but is not of finite rank. Then by the classification of irreducible Coxeter systems, $(W,S)$ is either of type $A_\infty$, $A_{\infty,\infty}$ as in \ref{ss:5.10}, or of type $B_{\infty}$, $D_{\infty}$ with Coxeter graphs as shown below, respectively: 
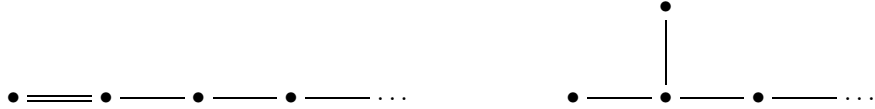
\[\xymatrix{&&&&&&&{\bullet}\ar@{-}[d]&\\
{\bullet}\ar@{=}[r]&\bullet\ar@{-}[r]&\bullet\ar@{-}[r]&\bullet\ar@{-}[r]&\ldots&&\bullet\ar@{-}[r]&\bullet\ar@{-}[r]&\bullet\ar@{-}[r]&\ldots}\] 

Corollary \ref{ss:5.17} and the classification of finite Coxeter systems  implies that for a based root datum 
with  associated  Coxeter system  $(W,S)$  of type $A_\infty$ (resp.,
$A_{\infty.\infty}$, resp., $B_\infty$, resp., $D_\infty$), 
 any abstract  root basis of $(W,S)$ is of type $A_{\infty}$ or $A_{\infty,\infty}$
(resp., $A_\infty$ or $A_{\infty,\infty}$, resp., $B_\infty$, resp., $D_\infty$).  

For use in the proof of \ref{ss:5.13}, we describe the various possible root bases in these cases.
 We realize   root systems of type  $A_{\infty,\infty}$ (resp., $A_{\infty}$) as  $E_S$ (resp., $E_K$) in \ref{ss:5.10}.  We realize the root system of type $B_\infty$
 as the subset \[\Phi_B:= \set{\pm (e_i+e_j),\pm(e_i-e_j), \pm e_i\mid i,j\in \Nat,i\neq j}\] and that of type  $D_\infty$ as the subset
\[ \Phi_D:=\set{\pm (e_i+e_j),\pm(e_i-e_j)\mid i,j\in \Nat,i\neq j}\]  of $V_K$ as in \ref{ss:5.10}.  
 As corresponding sets of simple roots, we take
 $\Pi_B=\set{e_0}\cup\mset{e_{n+1}-e_n\mid n\in \Nat}$, $\Pi_D=\set{e_0+e_1}\cup \mset{e_{n+1}-e_n\mid n\in \Nat}$. These determine standard based root datums $E'_B$, $E'_D$ of types $B_\infty$, $D_\infty$ respectively  on $V_K$
 with the restriction of the form $\mpair{?,?}$  on $V$ from \ref{ss:5.10}. Let $H$ be the ``infinite signed permutation group''  consisting of all invertible linear operators on $V_K$ which induce bijections on the set $\mset{\pm e_n\mid n\in \Nat}$.
 It is easy to check directly by ad hoc arguments  from the preceding paragraph  that, up to  sign, every root  basis of $E_S$ or $E_K$ is one of those described in the last paragraph of \ref{ss:5.10}. Further, the group $G(A_\infty)$ (resp., $G(A_{\infty,\infty})$ consisting of all $\real$-linear operators on $V_K$ (resp., $V_S$) which restrict to bijections of the set $\set{e_n\mid n\in \Nat}$ (resp., $\set{e_n\mid n\in \Int}$) acts transitively on the set of root bases of the root systems for   $E_K$ (resp.,  $E_S$) which are of type $A_{\infty}$ (resp., $A_{\infty,\infty}$).
  Similarly, $G(B_\infty):=H$ acts  transitively on the set of  root bases of  $\Phi_B$,  and $G(D_{\infty}):=H$ acts  transitively
   on the set of root bases of $\Phi_D$. 
   
   If $(W,S)$ is of type $A_{\infty}$ (resp., $A_{\infty,\infty}$, $B_{\infty}$, $D_{\infty}$)
   let $G$ be the group $G(A_{\infty})$ (resp., $G(A_{\infty,\infty})$, $G(B_\infty)$, $G(D_\infty)$) acting on the real root system $\Phi$ of corresponding type as above.
   Identify $\Phi=\widehat T$ as in \ref{ss:5.1}(b). From above, $G\times\set{\pm 1}$ acts transitively on the set of abstract root bases of $\widehat{T}$ of the same type ($A_\infty$, $A_{\infty,\infty}$, $B_\infty$ or $D_\infty$) as $(W,S)$. It is easy to see that the action of $G$ on $\widehat T$ is by elements of $\widehat W$; actually, there is a natural identification $G\cong\widehat{W}$ in each case.
    \end{example}

   \subsection{Proof of \ref{ss:5.13}} First, we prove (a)--(b). We assume for this  that $(W,S)$ is not of finite rank,
   otherwise (a)--(b) follows from Theorem \ref{ss:5.4}. If $(W,S)$ is of type (LocFin),  then (a)--(b) follow from
   \ref{ss:5.17}.    
   Now suppose that $(W,S)$ is of type (LocInf). Let $\epsilon$ and $a_{W'}$ be as in the statement of \ref{ss:5.16}.
   We define a map $\lambda\colon \Psi\rightarrow \Psi$ given locally by action of $W$, as follows.
   Let $R$ be a finite subset of $S_{\Psi_+}$ such that the Coxeter system $(\mpair{R},R)$ is infinite and  irreducible (note that the family of all such $R$ is cofinal in the family of all finite subsets of $S_{\Psi_+}$). Let $\Delta:=\mset{\alpha\in \Delta_{\Psi_+}\mid s_\alpha\in R}$. Choose  $J\subseteq S$ with $W_J\in\mc{S'}$ such that $\mpair{R}\subseteq W_J$.   Set 
   $\lambda(\alpha)=\bigl(a_{W_J} \bigr)^{-1}(\alpha)$ for all  $\alpha\in (\mpair{R}\cap T)\times \set{\pm 1}$.
   To show $\lambda$ is well-defined, it will suffice to show that if $J\subseteq K\subseteq S$ with
   $W_K\in \mc{S}'$, then $a_{W_K}a_{W_J}^{-1}$ fixes all   $\alpha\in (\mpair{R}\cap T)\times \set{\pm 1}$,
   or equivalently  that $p:=a_{W_K}^{-1}a_{W_J}$ fixes $a_{W_J}^{-1}(\Delta)$ elementwise.  But from the proof of \ref{ss:5.16},   $\epsilon a_{W_J}^{-1}\Delta=J'\times\set{1}\subseteq J\times\set{1}$ and $\epsilon a_{W_K}^{-1}\Delta=K'\times\set{1} \subseteq K\times\set{1}$ where $J'\subseteq J$, $K'\subseteq K$, so $p(J'\times \set{1})=K'\times\set{1}$.
   Since $(\mpair{R},R)$ is infinite and irreducible, so are $(W_{J'},J')$ and $(W_{K'},K')$ (since they are isomorphic to $(\mpair{R},R)$). By \ref{ss:conj} applied to $(W,S)$, $J'=K'$ and $p$ is in the subgroup of $W_K$ generated by reflections in $S\setminus J'$  which commute with each element of $J'$,  which gives the desired conclusion. Note that from the construction, $\epsilon\lambda(\Psi_+)\subseteq \widehat{T}_+$ and in fact, $\epsilon \lambda(\Delta_{\Psi_+})\subseteq S\times\set{1}$. 
   By symmetry, interchanging the roles of $\Psi_+$ and $\widehat T_+$, we also get a map $\lambda'\colon \Psi\rightarrow \Psi$ given locally by conjugation, and  with $\epsilon \lambda'(\widehat T')\subseteq \Psi_+$ and $\epsilon\lambda'(S\times\set{1})\subseteq \Delta_{\Psi_+}$. We claim that the composite $\lambda'\lambda$  is the identity map on $\Psi$.
   To see this, let $R$ and $\Delta$ be as above in the definition of $\lambda$. 
For some $w\in W$,    we have $\lambda'\lambda(\alpha)=w(\alpha)$ for all $\alpha\in (\mpair{R}\cap T)\times\set{\pm 1}$. Since $w(\Delta)\subseteq \Delta_{\Psi_+}$,  and $(\mpair{R},R)$ is infinite irreducible,
\ref{ss:conj} applied to $(W,S_{\Psi_+})$ shows that $w$ is a product of elements of $S_{\Psi_+}\setminus R$  which commute with each  element of  $R$, so $w(\alpha)=\alpha$ for all $\alpha\in (\mpair{R}\cap T)\times\set{\pm 1}$.
Similarly, $\lambda\lambda'=\Id_{\Psi}$ so $\lambda'=\lambda^{-1}$.
Hence, $\hat w:=\lambda\in \widehat W$.
    Clearly, $\epsilon \hat w $ maps $\Delta_{\Psi_+}$ bijectively to
   $S\times\set{1}$, maps $\Psi_+$ bijectively to $\widehat{T}_+$,   and induces an isomorphism $(W,S_{\Psi_+})\cong (W,S)$.  This proves (a) and (b) in case (LocInf), and hence in all cases.
   
   Now we prove (c) in general. Let $\Psi_+:=\sigma(\widehat{T})$. Then $\Psi_+$ is an abstract positive system and by \ref{ss:3.29},  $\sigma$ induces an isomorphism $s_\alpha\mapsto s_{\sigma(\alpha)}\colon (W,S)\cong (W,S_{\Psi_+})$. Let $\epsilon$, $\hat w$ be as in (b).
   Then $d'=\hat w^{-1}\epsilon'\sigma$ is an automorphism of $\widehat T$ fixing $\widehat T_+$ setwise,
   and therefore fixing the  standard root basis $S\times\set{1}$ setwise as well. It is clear that $d'$ induces 
   an automorphism $s_\alpha\mapsto s_{d'(\alpha)}$ of $(W,S)$ i.e. a diagram automorphism $d$  of $(W,S)$, and that $d'$ is the automorphism of $\widehat{T}$ induced by $d$. This completes the proof of the theorem.
  
\begin{corollary} Let $E$ be a root datum with root system $\Phi$ and $\Pi$, $\Pi'$ be two root bases for $E$
affording based root datums $B$, $B'$  with corresponding Coxeter systems
$(W,S)$ and $(W,S')$ respectively. Assume that $(W,S)$ and $(W,S')$ are isomorphic
and irreducible.
Then, after possibly rescaling $B'$ (still keeping $\Pi'\subseteq \Phi$),
there are linear maps $\sigma\colon V\rightarrow V$ and $\sigma'\colon V'\rightarrow V'$ and a sign
$\epsilon\in \set{\pm 1}$
with the following properties:
\begin{conds}
\item $\epsilon \sigma$ (resp., $\epsilon \sigma'$) restricts to  a    bijection $\Pi'\rightarrow \Pi$ (resp.,
$\ck \Pi\rightarrow \ck{\Pi'}$)
\item for any  subspaces  $V_1\subseteq V$ spanned by a finite set of roots and $V_1'\subseteq V'$ of $V'$ spanned by a finite set  of coroots, 
there is $w=w_{V_1,V_1'}\in W$ such that $\sigma(v)=w(v)$ for any $v\in V_1$ and
$\sigma(v')=w(v')$ for any $v'\in V_1'$   
\item $\sigma$ (resp., $\sigma'$) restricts to a permutation of $\Phi$ (resp., $\ck \Phi$) 
and  $\ck{\sigma(\alpha)}=\sigma'(\ck \alpha)$ for all $\alpha\in \Phi$.
\item $\mpair{\sigma(v),\sigma'(v')}=\mpair{v,v'}$ for all $v\in \real \Pi$, $v'\in \real\Pi'$.
\end{conds} In particular,
$\mpair{\alpha,\ck \beta}=\mpair{\sigma(\alpha),\ck{\sigma(\beta)}}$ for all $\alpha,\beta\in \Pi'$ i.e.  $\Pi$ and $\ck \Pi$ afford the same NGCM (up to rescaling and reindexing).
\end{corollary} 
\begin{remark} There is an obvious notion of an isomorphism of based root datums.
If $V=\real \Pi$ and $V'=\real \Pi'$,  $\epsilon\sigma$ and $\epsilon\sigma'$ together  define  such an isomorphism of based root datums $B\xrightarrow{\cong} B'$ (after the possible rescaling of $B'$).\end{remark}
\begin{proof} We assume without loss of generality that  $\Pi$ (and therefore also $\Pi'$) spans $V$,
and similarly for $V'$.
  If $(W,S)$ is of finite rank, the Corollary follows  from Theorem \ref{ss:5.3} with $\sigma=\sigma'\in W$. Assume now that $(W,S)$ is of infinite rank. If $(W,S)$ is of  type (LocInf), then, since any set of simple roots for a finite standard parabolic subgroup of $(W,S)$ must be linearly independent, $\Pi$ must be linearly independent.  Also, the Coxeter graph of $(W,S)$ is a tree.
Therefore, up  to rescaling $\Pi$, $B$ is of the type considered in  Examples \ref{ss:5.10} and \ref{ss:5.18}
and the desired conclusion follows readily.  Finally, assume $(W,S)$ is of infinite rank of type (LocInf).
We may define $\sigma$ following the definition of $\lambda$ in the proof of \ref{ss:5.13}. 
Namely, let $R\subseteq S'$ be such that the parabolic subsystem $(\mpair{R},R)$ is finitely generated and irreducible.   Consider the abstract positive $\Psi_+$ 
in $\widehat T$ corresponding to the positive roots of the root datum $B'$, and define $a_{W'}$, $\epsilon$ as in Lemma \ref{ss:5.16}. Following the proof of \ref{ss:5.13},
 set $\sigma(\alpha)=a_{W_J}^{-1}(\alpha)$ for all 
$\alpha$ in the $\real$-span of $\mset{\beta\in \Phi\mid s_\beta\in R}$. This is well-defined since for $p$, $J'$, $K'$ as in the proof of \ref{ss:5.13}, $p$ fixes $\mset{\alpha\in \Phi_+\mid s_\alpha\in J'}$ elementwise.  In a similar manner to that in the proof of \ref{ss:5.13}, one sees that $\sigma$ is invertible (here, as a linear map).
One may define $\sigma'$ in a similar way on $\real \ck \Pi$, and then clearly, (ii)--(iv) holds.  It is obvious  that $\epsilon \sigma(\Pi')$ differs from $\Pi$ simply by rescaling.
 Rescaling $\Pi$ appropriately, we therefore assume we have $\epsilon \sigma(\Pi')=\Pi$.
 One checks from (iii)--(iv) that this implies that  $\epsilon \sigma'(\ck{\Pi'})=\ck \Pi$, 
 and this gives (i).
 \end{proof}
 The analogue of the following for a class of crystallographic reflection representations
 (arising from Kac-Moody Lie algebras)  was proved in \cite{Moop}. 
 \begin{corollary} Suppose that $B$ is a based root datum affording an irreducible Coxeter system
 $(W,S)$ and that $V$ is finite dimensional. Then any two root bases of $B$ are $W$-conjugate up to sign and rescaling.\end{corollary}
 \begin{proof} This follows easily from  the preceding theorem on noting that
 $(W,S)$ cannot be of type (LocInf), hence is not of type $A_\infty$ or $A_{\infty,\infty}$\end{proof}

\appendix\section{Cocycles and extensions of group actions} 
\subsection{} \label{ss:1.1} Let $G$ be a group and $A$ be a (multiplicatively written) group on which $G$ acts on the left, so that $g(a\cdot a')=g(a)\cdot g(a')$. We recall (\cite[Appendix to Ch VII]{Serre})  that  a cocycle of $G$ in $A$  is a map $s\mapsto a_s\colon G \mapsto A$ satisfying
$a_{st}=a_s\cdot s(a_t)$ for all $s,t\in G$. Two cocycles $a$, $b$ are said to be  cohomologous if  there exists $c\in A$
such that $b_s=c^{-1}\cdot a_s\cdot s(c) $ for all $s\in G$. This defines an equivalence relation on the set of cocycles of $G$ in $A$, and the set $H^1(G,A)$ of equivalence classes is called the first cohomology set of $G$ with values in $A$. In general, $H^1(G,A)$ is only a pointed set
(with base point given by the equivalence class of the trivial cocycle $a$ defined by $a_s=1$ for all $s$).
If $A$ is abelian, $H^1(G,A)$ has  a natural structure of abelian group (with the base point as identity element). 

\subsection{} \label{ss:1.2} Let $S$ be a set with a left action of  the group $A$ denoted,  $(x,y)\mapsto x\cdot y$, for $x\in G,s\in S$) of $A$ and a
left action of the group  $G$, denoted  $(g,y)\mapsto g(y)$, compatible with the left action of $G$ on $A$ from \ref{ss:1.1} in the sense that $g(x\cdot y)=g(x)\cdot g(y)$ for $g\in G$, $x\in A$, $y\in S$.  This applies in particular with $S=A$, with its given left $G$-action and natural left $A$-action by left translation.  For a cocycle $a$ as in \ref{ss:1.1}, we let $S_a$ denote the $G$-set $S$  with the  left action $\times_a$ of $G$ on $A_a$ defined by
$s\times_a y=a_s\cdot s(y)$ for $s\in G$ and $y\in A$.   If $b$ is another cocycle, cohomologous to $a$
by the element $c\in A$ as in \ref{ss:1.1}, then the map $y\mapsto c\cdot y\colon S_b\rightarrow S_a$ defines an isomorphism of $G$-sets
 $S_b\rightarrow S_a$. 

\subsection{}  \label{ss:1.3}
Let $G$ be a group acting on a set $X$. For $x\in X$, let $[x]:=Gx$ denote the orbit of $x$. Assume given for each orbit $[x]$ a multiplicative group $A_{[x]}$. 
We consider a category $C$ in which an object is a $G$-set $\tilde X$ with a given (surjective) $G$-equivariant map
$\pi\colon \tilde X\rightarrow X$ and for each orbit $[x]$, a right  action of the group $A_{[x]}$ on $\pi^{-1}([x])$ commuting with the $G$-action on this set and such that for each $y\in [x]$, $A_{[x]}$ acts simply transitively on $\pi^{-1}(y)$. Morphisms are given by $G$-equivariant maps between the corresponding $G$-sets commuting with the projections $\pi$ and the $A_{[x]}$-actions on the inverse images
$\pi^{-1}([x])$, and composition of morphisms is by composition of the underlying  maps of $G$-sets.  (If there is a group $H$ with $A_{[x]}=H$ for all $x$, this notion is an analogue in the category of $G$-sets of a principal $H$-bundle \cite{Hus}).  Any morphism in $C$ is clearly an isomorphism. Below, we  indicate how the objects of $C$
are classified up to isomorphism  by an appropriate first cohomology group $H^1(G,A)$ 

For each $x\in X$, define  a multiplicative abelian group $A_x=A_{[x]}$. Form the product  group $A:=\prod_{x\in X} A_x$, with projections $\pi_x\colon A\rightarrow A_x$ for $x\in X$. For $a\in A$, we write $a$ as $a=(a_x)_{x\in X}$ or for short $a=(a_x)$, where $a_x=\pi_x(a)$. The group $G$ acts naturally on the left of  $A$ by $g(a)=b$ where $b_x=a_{g^{-1}(x)}$,
or for short $g(a_x)=(a_{g^{-1}(x)})$.

\begin{proposition} \label{ss:1.4}There is a natural bijection between the isomorphism classes of objects of the  category $C$ defined in $\text{\rm \ref{ss:1.3}}$ and $H^1(G,A)$.\end{proposition}
\begin{proof} Arbitrary functions $a\mapsto a_g\colon G\rightarrow A$ correspond bijectively to functions $\eta\colon G\times X\rightarrow \cup_{x\in X} A_{[x]}$  with
$\eta(g,x)\in A_x$ for all $x\in X$, $g\in G$, by the correspondence $a_g=(\eta(g,x))_{x\in X}$. 
One checks that the formula \[g(x,y_x):=(gx,\eta(g^{-1},x)^{-1}y_x)\] for $x\in X$, $g\in G$, $y_x\in A_x$
defines a left $G$-action on the set $\tilde X_a=\cup_{x\in X}\set{x}\times A_x$ iff the function
$g\mapsto a_g$ is a cocycle $a$ of $G$ with values in $A$.
In this case, $\tilde X_a$  has an obvious structure as an object of $C$. Clearly, every object of $C$ is isomorphic to one obtained from this construction from some function $\eta$.
It is straightforward to check that, for two cocycles $a$ and $b$, $\tilde X_a$ is isomorphic in $C$ to $\tilde X_b$ iff $a$ and $b$ are cohomologous. 
Hence the map sending  the cocycle $a$ to the isomorphism class of $\tilde X_a$ gives the required bijection. 
 \end{proof}
\subsection{}  Although we shall not need it, we also describe  an analogue in the above setting  of the construction of an associated bundle of a principal bundle.
For $x\in X$, let $S_x$ be a left $A_{[x]}$-set with $S_x=S_y$ whenever $[x]=[y]$.
Give the product set $S:=\prod_{x\in X}S_x$ the left $A$-action   with $(a_x)\cdot (y_x)=(a_x\cdot y_x)$. There is an action of
$G$ on $S$  defined by $g(y_x)=(y_{g^{-1}(x)})$, compatible with the $G$-action on $A$ in the sense of \ref{ss:1.2}.   If $a\colon G\rightarrow A$  is a cocycle and $\eta$ the corresponding function as in the proof of \ref{ss:1.4}, the formula
 \[g(x,y_x):=(gx,\eta(g^{-1},x)^{-1}y_x)\] for $x\in X$, $g\in G$, $y_x\in S_x$
defines a left $G$-action on the set $X':=\cup_{x\in X}\set{x}\times S_x$. 

Any  $G$-set $X'$ with given equivariant map $\pi'\colon X'\rightarrow X$ such that for each $G$-orbit $[x]$, there is some set $S_{[x]}$ such that  the fibers $(\pi')^{-1}(y)$ for $y\in [x]$ are all in bijection with  $S_{[x]}$, arises from the above construction with $A_{[x]}=\text{\rm Sym}(S_{[x]})$, the symmetric group  acting on $S_{[x]}$ in the usual way,  and $S_x:=S_{[x]}$.

 \section{Quasi-root systems}
In this appendix, we show how the definition and some elementary properties of Bruhat order 
and weak order on Coxeter groups extend to groups with a suitable  quasi-root system realized linearly in  a real vector space. The principal example other than Coxeter groups is provided by real orthogonal groups.

\subsection{}  \label{ss:2.6} Let $(\Psi,F)$ be a quasi-root system. We use notation as in \ref{ss:2.2}. 
Suppose that $\Psi_+$ is a quasi-positive system for $\Psi$, and let $N:=N_{\Psi_+}\colon W\rightarrow \mc{P}(T)$ be the corresponding  reflection cocycle.

We define certain pre-orders (reflexive, transitive relations) on  $W$ as follows.
The weak pre-order $\leq_w$ is defined by $x\leq_w y$ iff $N(x)\subseteq N(y)$; this is a partial order iff $N(x)=N(y)$ for $x,y\in W$ implies $x=y$, or equivalently by the cocycle condition, iff $N(x)=\emptyset$ with $x\in W$  implies $x=\emptyset$. 
Let $A\subseteq T$. The twisted Bruhat pre-order $\leq_A$ on $W$ is defined by 
$x\leq_A y$ iff there exist
$x=x_0,x_1,\ldots, x_n\in W$ and $t_i\in x_i\cdot A$ with $ x_{i-1}=t_ix_i$,
where $w\cdot A=N(w)+wAw^{-1}$.
It is easy to see that for $x,y\in W$, 
 \begin{equation}\label{eq:2.6.1} x\leq_Ay\Longleftrightarrow xw^{-1}\leq_{w\cdot A}yw^{-1}\Longleftrightarrow y\leq _{T+A}x\end{equation}
For $A=\emptyset$, the order $\leq_A$ is denoted just as $\leq$ and is  called  the Bruhat pre-order on $W$. 

In the case $(\Psi,F)$ is the standard  abstract root system of a Coxeter system $(W,S)$, 
$\leq_w$ and $\leq$ are partial orders, called weak right order and Bruhat order respectively, and $\leq_A$ is a twisted Bruhat order in the sense of \cite{Dy1} if $A$ is an ``initial section of a reflection order'' in the sense of loc cit.

\subsection{} We define a linear realization of 
$(\Psi,F)$ to be a tuple $(\mpair{?,?},R,M,\ck M,\iota,\ck\iota)$ where
\begin{conds} \item $R$ is a ring,  $M$ (resp., $\ck M$) is a left (resp., right $R$-module)
and  $\mpair{?,?}\colon M\times \ck M\rightarrow R$ is a $R$-bilinear form
\item $\iota\colon \Psi\rightarrow M$ and  $\ck\iota\colon \Psi\rightarrow \ck M$ are injective functions.    We identify
$\Psi$ with $\iota(\Psi)$,  so $\iota$ becomes an inclusion, set $\ck \Psi:=\ck\iota(\ck \Psi)$ and set $\ck \alpha=\ck\iota (\alpha)$ for $\alpha\in \Psi\subseteq M$ so $\nu:=(\alpha\mapsto \ck\alpha)\colon \Psi\mapsto \ck\Psi$ is a bijection. 
\item $\mpair{\alpha,\ck \alpha}\alpha=2\alpha$, $\ck\alpha\mpair{\alpha,\ck\alpha}=2\ck \alpha$ for $\alpha\in \Psi$.
For $\alpha\in  \Psi$,  define $s_\alpha\in\text{\rm GL}_R(M)$ by $m\mapsto m-\mpair{m,\ck\alpha}\alpha$ and $s_{\ck\alpha}\in \text{\rm GL}_R(\ck M)$ by
$m\mapsto m-\ck \alpha\mpair{\alpha, m}$
\item for $\alpha\in \Psi$, $s_\alpha(\Psi)=\Psi$ and  $F(\alpha)=(s_\alpha)_{\vert \Psi}$,
while  $s_{\ck \alpha}(\ck\Psi)=\ck \Psi$  and $(s_{\ck\alpha})_{\vert \ck \Psi}= F(\alpha) $
\item $\mpair{s_\alpha\mid \alpha\in \Psi}$ identifies with $W\subseteq \text{\rm Sym}(\Psi)$ by restriction and $\mpair{s_{\ck \alpha}\mid \alpha\in \Psi}$ identifies with $ W \subseteq \text{\rm Sym}(\ck\Psi)$ by restriction 
\end{conds} 
Here, we use $\nu$ to identify any action of a group, such as $W:=\mpair{F(\alpha)\mid \alpha\in \Psi}$, on $\Psi\subseteq M$ with an action of that group
on $\ck\Psi$.   
Note (v) follows  automatically from (i)--(iv)  if $R\Psi=M$ and $\ck \Psi R=\ck M$.

\subsection{} 
The main example of  quasi-root systems and their  linear realizations are provided by
Coxeter groups. Any  root  datum $E$ as in \ref{ss:4.5} may be naturally regarded as providing a linear realization  of some quasi-root system (in fact, an abstract root system in the sense of \ref{ss:3.5b}). Other  interesting  examples  of linear realizations 
of   abstract root systems arise from natural  reflection representations of Coxeter groups   over
certain commutative or  non-commutative coefficient rings.

\subsection{} \label{ss:2.8} Certain orthogonal groups  provide examples of (linearly realized) quasi-root systems as follows. Let $K$ be a field of characteristic unequal to two, and 
   $V$ be a finite-dimensional vector space over  $K$ equipped with a symmetric, non-degenerate  bilinear form $(?\mid ?)\colon V\times V\rightarrow K$.
   Let $O(V)$ be the orthogonal group of $(V,(?\mid?))$, consisting of invertible $K$-linear transformations of $V$ preserving the form. For non-isotropic $\alpha\in V$,  let $s_\alpha\in O(V)$ denote the orthogonal reflection, defined by  $ v\mapsto v-2(v\mid \alpha)/(\alpha\mid\alpha)\alpha$. We choose a subset $\Psi$ of $V$  such that $\Psi$ is stable under $O(V)$, consists of non-isotropic vectors 
   and each non-isotropic line $l$ in $O(V)$ contains at least one element of $\Psi$.  Such a set $\Psi$ always exists; for example, one could take $\Psi$ to consist of all non-isotropic vectors in $V$.
   Alternatively, one could choose $\Psi$ so that each non-isotropic line contains exactly two elements of $\Psi$; for instance, if  $K=\real$, one could take $\Psi$ to consist of all $\alpha\in V$ with $(\alpha\mid\alpha)=\pm1$. For $\alpha\in \Psi$, let $F(\alpha)\in \text{\rm Sym}(\Psi)$ denote the restriction of $s_\alpha$ to $\Psi$. Let $\ck\Psi:=\mset{\ck\alpha\mid \alpha\in \Psi}$ where $\ck \alpha:=2\alpha/(\alpha\mid \alpha)$.
   
Then clearly  $(\Psi,F)$ is a quasi-root system. Moreover,   $((?,?),K, V,V,\iota,\ck\iota)$
  is a linear realization of $(\Psi, F)$ where $\iota\colon \Psi\rightarrow V$ is the inclusion  and $\ck\iota\colon \Psi\rightarrow V$ is  the map
$\alpha\mapsto\ck \alpha$. 
\subsection{} Assume that  $(\Psi,F)$ is a quasi-root system  with a linear realization  over $\real$ of the form 
$(\mpair{?,?},\real,V,\ck V,\iota,\ck\iota)$, so $V$ and $\ck V$ are real vector spaces.
We fix a vector space total ordering $\preceq$ of $\ck V$ i.e. a total ordering of $\ck V$
such that the set of positive elements is closed under addition and multiplication by positive real numbers. We also assume given a  family $\set{\omega_i}_{i\in I}$ of elements of $\ck V$, indexed by a well-ordered set $I$,  such that if $v\neq 0$ in $V$, there is $i\in I$ with $\mpair{v,\omega_i}\neq 0$.

For a totally ordered set $X$ with total order $\leq$, let $X^I$ be the family
of $I$-tuples $(x_i)_{i\in I}$ of elements of $X$, totally ordered by the lexicographic order $\leq_\lex$ 
induced by $\leq$ on $X$ and the well-ordering  of $I$. In particular, we have totally ordered sets
$(\ck V)^I$ with order $\preceq_\lex$ and $\real^I$ with order $\leq_\lex$  defined using the usual ordering $\leq$  of
$\real$. The vector space embedding $v\mapsto (\mpair{v,\omega_i})_{i\in I}\colon V\rightarrow  \real^I$ gives rise, by restriction of the order $\leq_{\lex}$ , to a vector space total ordering of $V$, which we denote also as $\leq_\lex$. (It is well-known that if $V$ is finite dimensional and $\mpair{?,?}$ is a perfect pairing, then every vector space total order of $V$ is equal to an order $\leq_{\lex}$ arising from some family $\set{\omega_i}$).

It is easy to see that for $\alpha,\beta\in \Psi$, $s_\alpha=s_\beta$ iff there is $c\in \real_{\neq 0}$ such that $\beta=c\alpha$ and $\ck\beta=c^{-1}\ck \alpha$; it follows readily that  $\Psi_+:=\mset{\alpha\in \Psi\mid 0\preceq \ck\alpha}$
and $\Phi_+:=\mset{\alpha\in \Psi\mid \alpha \leq_{\lex} 0}$ are  compatible quasi-positive systems
for $(\Psi,F)$. Define the reflection cocycle
$N=N_{\Psi_+}\colon G\rightarrow \mc{P}(T)$, and set  $A=\mset{F(\alpha)\mid\alpha\in \Phi_+\cap -\Psi_+}\subseteq T$.
\begin{proposition}  Let assumptions and notation be as immediately above. Then: \begin{num}\item  For all $\alpha\in \Psi_+$ and $w\in W$, \[(s_\alpha w(\omega_i))_{i\in I}\preceq_\lex (w(\omega_i))_{i\in I}\] in $(\ck V)^I$ iff $s_\alpha\in w\cdot A=N(w)+wAw^{-1}$.
\item The  twisted Bruhat pre-order  $\leq_A$ is a partial order on $W$. 
 \end{num}
\end{proposition}
\begin{proof} For (a), we have $(s_\alpha w (\omega_i))_i\preceq _\lex (w(\omega_i))_i $
iff $(w(\omega_i)-\mpair{\alpha, w\omega_i}\ck \alpha)_i\preceq_\lex  (w(\omega_i))_i$ in $(\ck V)^I$
iff $(\mpair{\alpha, w(\omega_i)})_i\geq_{\lex} (0)_i$ in $\real^I$ (since $0\prec \ck \alpha$) 
iff $w^{-1}(\alpha)\in -\Phi_+$ iff $s_\alpha\in w\cdot A$ by \ref{ss:2.5}(a).
Now we show that (a)  implies that $\leq_A$ is transitive, which will prove (b).
If $x\leq_A y$ and $y\leq _A x$, then by (a) we have 
\[(  x(\omega_i))_i\preceq _\lex (y(\omega_i))_i \preceq_\lex x (\omega_i))_i\]
and hence $(x(\omega_i))_i= (y(\omega_i))_i$. However, if $x\neq y$, we have $x^{-1}(v)\neq y^{-1}(v)$ for some $v\in V$, so $\mpair{x^{-1}(v),\omega_j}\neq  \mpair{y^{-1}(v),\omega_j}$ for some $j\in I$ by the assumption on the family $(\omega_i)_i$. But then $\mpair{v,x(\omega_j)}\neq
\mpair{v,y(\omega_j)}$ and $x(\omega_j)\neq  y(\omega_j)$ contrary to the above.\end{proof}

   \subsection{} We conclude with some remarks about the above-defined orders in the  special case  
   of the quasi-root systems of  real orthogonal groups  defined in \ref{ss:2.8}.
   In this situation, every vector space total order on $V=\ck V$ arises as the order $\leq_\lex$ from some family   $(\omega_i)_{i\in I}=(\omega_1,\ldots ,\omega_n)$ which may without loss of generality be taken to be a basis of $V$. It is easy to see that $w(\Psi_+)=\Psi_+$ implies that $w=1$, and it follows that the weak pre-order  is actually  a partial order.  Since $-1\in W$ with $N(-1)=\Psi_+$, one readily sees that $w\mapsto -w$ is an isomorphism $(W,\leq_A)\cong (W,\leq_A\op)$ in any of the twisted Bruhat orders $\leq _A$ and in the weak order $\leq_w$ (this is the analogue of the fact that multiplication by the longest element induces an order-reversing bijection of a finite Coxeter group in (any twisted)  Bruhat order or weak order).
   However, there may be several non-isomorphic Bruhat orders (and several non-isomorphic weak orders) depending on the choices of  $\preceq$, $(\omega_i)_i$ etc.

    Assume now that the form $(?\mid?)$ on $V$   is positive definite. Then, without loss of generality,  one may take $(\omega_i)_i$ as an orthonormal basis of $V$, and one sees that  the vector space total orders of $V$ correspond bijectively 
  to ordered orthonormal bases $(\omega_1,\ldots \omega_n)$ of $V$. In particular, 
  $W=O(V)$ acts simply transitively on the set of vector space total orderings of $V$, and the choice of
  the particular ordering $\preceq$ will not effect the family of  order types of posets arising as (W,$\leq_A)$ for varying $A$ (resp., the order type of  $\leq_w$). Further, from this one sees that $\Phi_+=w(\Psi_+)$ for some $w\in W$.
  Hence using \ref{ss:2.5e} and \eqref{eq:2.6.1}, the posets $(W,\leq_A)$ are all isomorphic to $(W,\leq_\emptyset)$.
  Thus, for a positive definite form, there is, up to poset isomorphism, only one twisted Bruhat order $\leq _A$ on $W$ (this is analogous to  the fact that there is, up to isomorphism, only one twisted Bruhat order on a finite Coxeter group).

One might ask whether  other properties of Bruhat order and weak order on finite Coxeter groups are shared by the corresponding orders on orthogonal groups.  For example, does $W$ under the weak
order form a (complete) lattice?

\end{document}